\documentclass[11pt,reqno]{amsart}
\usepackage[leqno]{amsmath}
\usepackage{amscd,amssymb,palatino}
\usepackage[english]{babel}
\usepackage{caption}
\usepackage{etex}
\usepackage[leqno]{amsmath}%
\usepackage{amssymb}
\usepackage{color}
\usepackage{shadow}
\usepackage{epsfig}
\usepackage{epic}
\usepackage[dvipsnames]{xcolor}
\usepackage{graphics}
\usepackage{graphicx}
\usepackage{psfrag}
\usepackage{url}
\usepackage{hyperref}
\usepackage{bm}
\usepackage{comment}
\usepackage[skip=5pt, indent=40pt]{parskip}

\usepackage {tikz}
\usetikzlibrary {automata, positioning}
\definecolor {processblue}{cmyk}{0.96,0,0,0}

\usepackage{cleveref} 
\usepackage{tensor}
\usepackage{enumitem}

\usepackage{amsthm, amscd, amsfonts, amssymb} %
\usepackage{mathtools} %
\usepackage{color}
\usepackage[all]{xy}
\usepackage{multicol}

\theoremstyle{plain}
\newtheorem{theoremIntro}{Theorem}

\newtheorem{proposition}{Proposition}[section]
\newtheorem{theorem}[proposition]{Theorem}
\newtheorem{corollary}[proposition]{Corollary}

\newtheorem{lemma}[proposition]{Lemma}

\theoremstyle{definition}
\newtheorem{definition}[proposition]{Definition}
\newtheorem{example}[proposition]{Example}

\theoremstyle{remark}
\newtheorem{remark}[proposition]{Remark}

\newcommand\restr[2]{{
  \left.\kern-\nulldelimiterspace
  #1
  \vphantom{\big|}
  \right|_{#2}
  }}

\DeclarePairedDelimiterX{\innerp}[2]{\langle}{\rangle}{#1,#2}

\newcommand{\NN}{\mathbb{N}}
\newcommand{\ZZ}{\mathbb{Z}}

\newcommand{\HH}{\mathbb{H}}

\newcommand{\RR}{\mathbb{R}}

\newcommand{\DD}{\mathbb{D}}

\newcommand{\cG}{\mathcal{G}}

\newcommand{\cat}[1]{\textup{\bfseries #1}}

\newcommand{\Bord}{\cat{Bord}}

\newcommand{\PF}{\cat{PF}}
\newcommand{\PRF}{\cat{PRF}}

\newcommand{\Hom}{\textup{Hom}}

\usepackage{setspace}
\usepackage{graphicx}

\usepackage{tikz-cd}
\usepackage{circuitikz}
\usetikzlibrary{shapes.geometric, arrows}

\tikzstyle{mytheorembox} = [draw=vdgreen, fill=blue!20, very thick, rectangle, rounded corners, inner sep=10pt, inner ysep=15pt]
\tikzstyle{mytheoremfancytitle} =[fill=vdgreen, text=white]

\definecolor{vdblue}{rgb}{0,0,.3}
\definecolor{dblue}{rgb}{0,0,.7}
\definecolor{lblue}{rgb}{.3,.3,1}
\definecolor{vlblue}{rgb}{.7,.7,1}
\definecolor{vvlblue}{rgb}{.9,.9,1}

\definecolor{vdred}{rgb}{.3,0,0}
\definecolor{dred}{rgb}{.7,0,0}
\definecolor{lred}{rgb}{1,.3,.3}
\definecolor{vlred}{rgb}{1,.7,.7}

\definecolor{vdgreen}{rgb}{0,.2,0}
\definecolor{dgreen}{rgb}{0,.4,0}
\definecolor{lgreen}{rgb}{.3,1,.3}
\definecolor{vlgreen}{rgb}{.7,1,.7}

\definecolor{lyellow}{rgb}{1,1,.3}

\definecolor{gray1}{rgb}{0.22,0.22,0.22}
\definecolor{gray2}{rgb}{0.28,0.28,0.28}
\definecolor{gray3}{rgb}{0.36,0.36,0.36}
\definecolor{gray4}{rgb}{0.44,0.44,0.44}
\definecolor{gray5}{rgb}{0.52,0.52,0.52}
\definecolor{gray6}{rgb}{0.6,0.6,0.6}
\definecolor{gray7}{rgb}{0.68,0.68,0.68}
\definecolor{gray8}{rgb}{0.76,0.76,0.76}

\definecolor{color1}{rgb}{1,0,0}
\definecolor{color2}{rgb}{0.98,0,0.816}
\definecolor{color3}{rgb}{0.717,0,1}
\definecolor{color4}{rgb}{0,0,1}
\definecolor{color5}{rgb}{0,1,1}
\definecolor{color6}{rgb}{0,1,0}
\definecolor{color8}{rgb}{1,1,0}
\definecolor{color7}{rgb}{1,0.651,0}

\makeatletter
\let\old@enddoc@text\enddoc@text

\renewcommand{\enddoc@text}{%
  \ifx\@setthanks\@empty\else
    \par
    \footnotesize
    \skip@20\p@ \advance\skip@-\lastskip \advance\skip@-\baselineskip
    \vskip\skip@
    \@setthanks
  \fi
}

\makeatletter
\newcommand{\printauthorinfo}{%
  \par\vspace{1em}  %
  \begingroup
  \footnotesize
  \@setaddresses
  \endgroup
}
\makeatother

\begin{document}
\title{Topological Kleene Field Theories as a model of computation}



\author{Ángel González-Prieto}
\address{ \'Angel Gonz\'alez-Prieto, Facultad de CC.\ Matem\'aticas, Universidad Complutense de Madrid, 28040 Madrid, Spain and Instituto de Ciencias Matem\'aticas (CSIC-UAM-UCM-UC3M),
28049 Madrid, Spain.}
\email{angelgonzalezprieto@ucm.es}
\author{Eva Miranda}\address{Eva Miranda,
Laboratory of Geometry and Dynamical Systems \& SYMCREA, Department of Mathematics, EPSEB, Universitat Polit\`{e}cnica de Catalunya-IMTech
in Barcelona and
\\ CRM Centre de Recerca Matem\`{a}tica, Campus de Bellaterra
Edifici C, 08193 Bellaterra, Barcelona.
 }%
 \thanks{Corresponding author: Eva Miranda. Email: \texttt{eva.miranda@upc.edu}}
 \email{eva.miranda@upc.edu}
 \author{Daniel Peralta-Salas} \address{Daniel Peralta-Salas, Instituto de Ciencias Matem\'aticas, Consejo Superior de Investigaciones Cient\'ificas,
28049 Madrid, Spain.}
\email{dperalta@icmat.es}

\begin{abstract} In this article, we establish the foundations of a computational field theory, which we term \textit{Topological Kleene Field Theory} (TKFT), inspired by Stephen Kleene’s seminal work on partial recursive functions and drawing parallels with Topological Field Theory. Our central result shows that any computable function can be simulated by the flow on a smooth bordism of a vector field with good local properties, setting an alternative model of computation to Turing machines. 
We thus establish that a computable function can be fully realized within a single go of a dynamical system, differing from previous works where computation is encoded as an iterative process. The output of the computable function emerges directly, laying the groundwork for potential applications that accelerate the physical realization of computation.

\end{abstract}

\maketitle

\section{Introduction}

The theory of computation has traditionally been grounded in discrete models, from Turing machines to lambda calculus, each foundational to understanding the limits and potential of computable functions. Yet, as computational complexity and dynamics interact across fields like fluid mechanics and quantum physics, there is a growing interest in bridging discrete computation with continuous systems. In this direction, one of the richest facets of computability theory is the existence of multiple models of computation. Over the last century, various definitions of computation have been introduced in the literature, such as Turing Machines \cite{turing1936computable}, $\lambda$-calculus \cite{church1932set}, combinatory logic \cite{curry1930grundlagen}, or counter machines \cite{shepherdson1963computability}, among others. Despite the significant differences in their nature, all these models can be shown to have equivalent computational power. This equivalence forms the basis of the Church-Turing thesis, a philosophical principle asserting that all models of computation are, at most, as powerful as a Turing machine.

One of the most interesting proposals in the model of partial recursive functions, as introduced by Kleene in \cite{kleene1936general}, is that computability must be understood as output functions for Turing machines. These are the mathematical abstractions of computers, comprised of a reading/writing head that manipulates a tape that works as an auxiliary memory. In this manner, a computable function is a partial function $f: \NN \dashrightarrow \NN$ that can be calculated by executing a Turing machine. Kleene proved that these computable functions coincide with partial recursive functions, which can be easily described as composition of some basic and explicit functions.

In this work, we show that it is possible to represent any computable function as the flow of a smooth vector field. Specifically, given such $f: \NN \dashrightarrow \NN$, we construct a smooth bordism $W_f$ with corners between two $2$-dimensional disks and a vector field $X$ transverse to the disks and tangent to the rest of the boundary (see Figure \ref{fig:bordism-dynamics}) such that the reaching function of the flow of $X$ at the boundary coincides with $f$.

\begin{center}
    \includegraphics[width=0.4\linewidth]{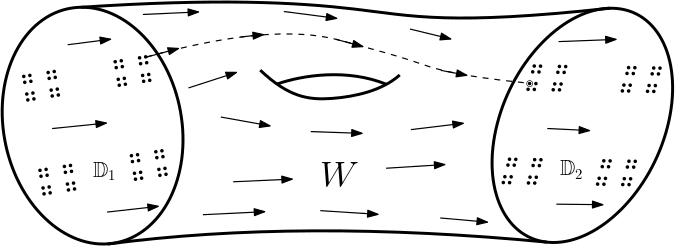}
    \captionof{figure}{Dynamical bordism with generated flow.}
    \label{fig:bordism-dynamics}
\end{center}

To be precise, on each boundary disk $\DD$, we encode a copy of the natural numbers through an embedding $\NN \hookrightarrow \DD$, for instance, by encoding every natural number according to its binary expression, as in a square Cantor set. This defines a partial function $Z_0(W, X): \NN \dashrightarrow \NN$, called the `reaching function', as follows: Given $n \in \NN$ in the input disk, we consider the flow of $X$ starting at $n$; if the flow hits $\DD_2$ at a point $n' \in \NN$ of the output disk, we set $Z_0(W, X) (n) = n'$, otherwise the result is undefined.

\subsection*{Related work}

The idea of investigating the interplay between dynamics and computation has been widely explored in the literature since the pioneering works of Moore~\cite{moore1990unpredictability,moore1991generalized}. In this direction, dynamical systems have been shown to exhibit computational abilities, such as solving linear programming problems \cite{brockett1991dynamical}, undecidability in classical mechanics \cite{da1991undecidability}, the non-computability of Julia sets \cite{braverman2006non}, the simulation of $\lambda$-calculus using a chemical-inspired machine \cite{berry1989chemical}, the construction of Turing-complete cellular automata \cite{wolfram1983statistical,wolfram2003new,wolfram2019cellular}, or the encoding of Turing machines in reaction-diffusion models \cite{bandini2005computing}. This latter approach of embedding Turing machine dynamics into particular dynamical systems has been proven useful in understanding the complexity of the solutions to differential equations, as in the context of analytic or polynomial families of ODEs~\cite{koiran1994computability,koiran1999closed,gracca2008computability,bournez2013computation,gracca2023analytic,gracca2024robust,gracca2005robust}, in \cite{tao2017universalitywell} for a potential well, or \cite{tao2017universality,tao2019universality} for the Euler equations modeling the dynamics of an ideal fluid flow.

Inspired by these ideas, T.\ Tao suggested that the complexity of the Navier-Stokes equations modeling realistic fluids could be studied by understanding the encoding of universal Turing machines as part of the flow. This idea has been exploited further in \cite{CMPP,CMP1,CMP2,CMPP2}, where it was proven that such universal Turing machines can be embedded within the flow of a steady state of the Euler equations in a $3$-dimensional manifold, or in \cite{dyhr2025turing} for steady states of the Navier–Stokes equations via cosymplectic geometry.

However, despite the interest of these works, the approach taken in this paper is deeply different: we do not aim to encode a Turing machine as a flow, which would require to be iterated many times to compute the final result; instead, we prove that the actual computable function can be simulated using a dynamical system in such a way that the outcome of the Turing machine can be represented in a single go of the flow.

\subsection*{Main result} The main theorem established in this work is the following.

\begin{theoremIntro}\label{thm:main-theorem-intro}
    For any partial recursive function $f: \NN \dashrightarrow \NN$, there exists a clean dynamical bordism $(W,X)$ between disks, such that $X$ is a volume-preserving vector field and $Z_0(W, X) = f$. Conversely, the reaching function of any clean dynamical bordism is a partial recursive function. Accordingly, the category of clean dynamical bordisms is equivalent to that of computable functions. 
\end{theoremIntro}

{In this theorem, ``dynamical bordism'' means that the vector field $X$ points inwards at the input disk and outwards at the output disc, and it is tangent to the rest of the boundary. The ``cleanness'' property is a technical condition that ensures that the vector field is ``well-behaved'' in the sense that its zero set consists of finitely many points, and there is a ``skeleton'' where the flow of $X$ is generated by simple diffeomorphisms of disks (see Section \ref{sec:pf-bordisms} for precise definitions). 

It is worth noticing that the construction of Theorem \ref{thm:main-theorem-intro} does require a non-trivial topology in the bordism to capture any partial recursive function. If we restrict ourselves to the class of bordisms obtained by taking the suspension of a diffeomorphism $\varphi: \mathbb{D} \to \mathbb{D}$ into its mapping cylinder $(W_\varphi, X_\varphi)$, the reaching map is exactly $\varphi|_{\NN}$ and thus it is continuous in the subspace topology for $\NN \subseteq \DD$. But not all the computable functions are continuous in this topology. The reaching map of a general bordism may instead be discontinuous, providing the needed flexibility.
}

\subsection*{Sketch of the proof} For the proof of Theorem \ref{thm:main-theorem-intro}, we deeply exploit an alternative representation of a Turing machine as a finite state machine. In this way, given a Turing machine $M_f$ computing a partial recursive function $f: \NN \dashrightarrow \NN$, we express $M_f$ as a directed finite graph where its nodes correspond to the states of $M_f$ and the edges represent the simple operations performed by the machine at each step. Encoding the tape as a point in the disk through a Cantor set-like embedding, we prove that these simple operations can be implemented using an explicit diffeomorphism. 

Hence, from the graph representing $M_f$, we create a CW-complex $W_f$ by placing a disk for each node and a mapping cylinder for each edge, representing the operation performed by the edge, as depicted in Figure \ref{fig:thickening}. Using the natural vector field produced by a homotopy between the identity and each of the edge diffeomorphisms, we endow the CW-complex with a vector field. Then, by smoothing $W_f$ into a homotopically equivalent smooth manifold and doing a suitable (volume-preserving) extension of the flow, we get the desired bordisms whose dynamics represent the computable function $f$.
A detailed proof of this result is provided in Section \ref{sec:prf-quantizable} and Theorem \ref{thm:clean-kleen}.

\begin{center}
    \includegraphics[width=0.4\linewidth]{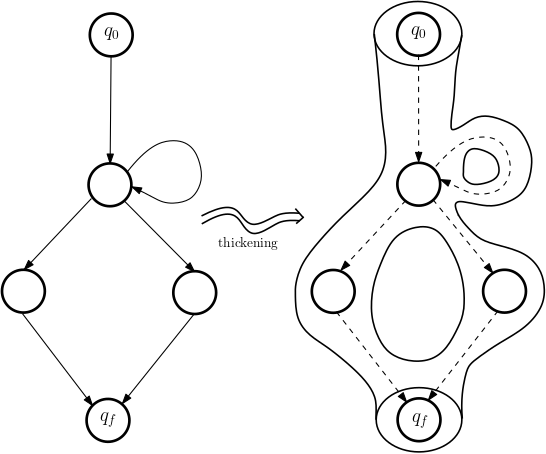}
    \captionof{figure}{Thickening process to represent a partial recursive function as a dynamical bordism.}
    \label{fig:thickening}
\end{center}
    
We underscore that the thickening construction used in the proof of Theorem \ref{thm:main-theorem-intro} is fully explicit and amenable, enabling effective simulation; cf. \cite{Ramos_2025}. In this direction, Figure \ref{fig:simulation} shows a dynamical bordism simulating a partial recursive function, and different trajectories corresponding to four different inputs, depicted in different colours.
The planes inside the bordism serve as computational checkpoints, each corresponding to a state of the underlying Turing machine.

\begin{center}
    \includegraphics[width=0.35\linewidth]{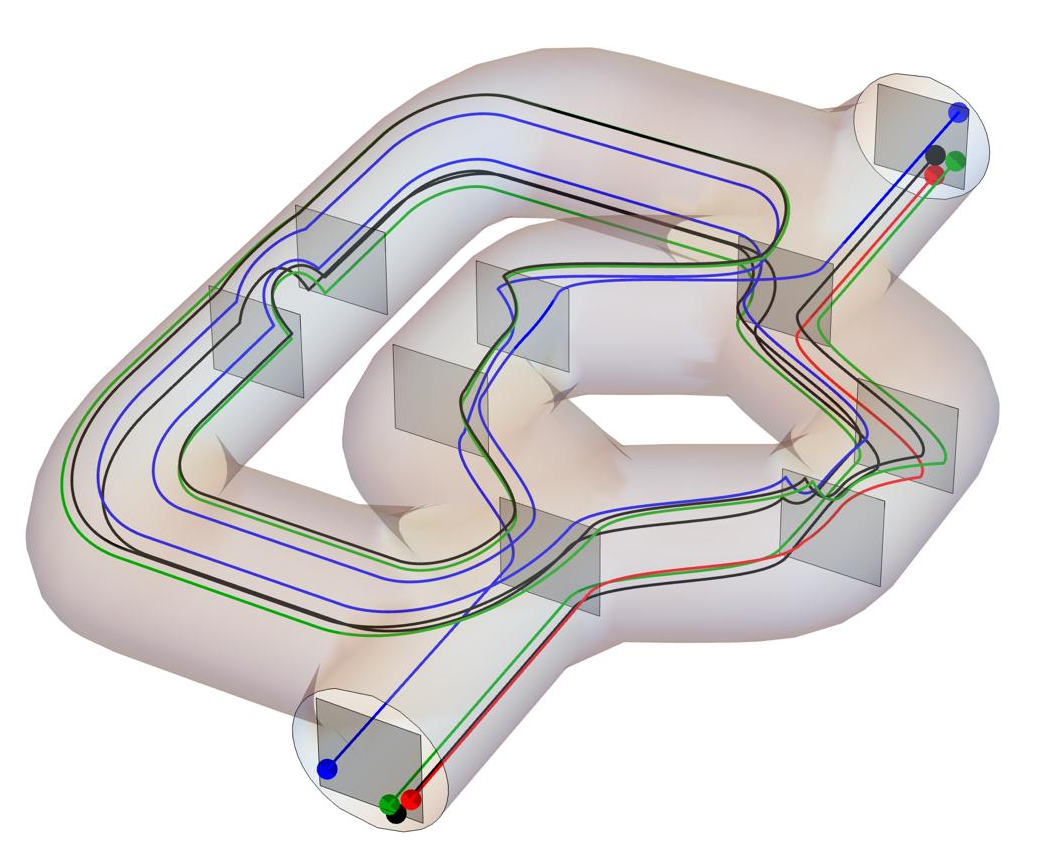}
    \captionof{figure}{Simulation of a partial recursive function through a dynamical bordism, reproduced from \cite{Ramos_2025}.}
    \label{fig:simulation}
\end{center}

\subsection*{Implications} Theorem \ref{thm:main-theorem-intro} has several consequences in the theory of dynamical systems. For instance, it shows that some problems in dynamics, such as determining whether the flow of a point will reach another point or even if the flow will be trapped, are undecidable due to the halting problem. Additionally, Theorem \ref{thm:main-theorem-intro} means that the class of volume-preserving vector fields is Turing complete in the sense that any computable function can be simulated by the flow of a volume-preserving field on a bordism as above.

It is worth mentioning that Theorem \ref{thm:main-theorem-intro} can be reinterpreted from a categorical perspective, providing a computational analogue of Topological Quantum Field Theories (TQFTs) \cite{witten1988topological,atiyah1988topological,witten1989quantum}, as we shall explore in Section \ref{sec:categorical-interpretation}. Recall that TQFTs are a mathematical formulation of quantum field theories independent of the underlying metric purely in terms of category theory by using the category of bordisms. In this setting, we show that the reaching function $Z_0$ leads to a functor of $2$-categories between the category of dynamical bordisms and the category of partial functions. Due to this parallelism with Topological Quantum Field Theories, we decided to refer to the construction of Theorem \ref{thm:main-theorem-intro} as a \emph{Topological Kleene Field Theory} (TKFT for short). Indeed, in his ``Lectures on Computation'', Richard Feynman delves into the algebraic structures underpinning quantum computing, drawing parallels to TQFT.

TKFTs represent a new model of computation in which computable functions are central objects that can be executed by following the flow of a vector field. Under this point of view, it opens new perspectives in complexity theory, since the computational complexity of the function can be related to the topological complexity of the bordism and the vector field. %
Furthermore, even though the proof developed in this work to establish Theorem \ref{thm:main-theorem-intro} requires the use of an auxiliary Turing machine, there may exist other clean dynamical bordisms with the same properties, but whose dynamics are constructed using different techniques. These non-Turing bordisms could provide a way of speeding-up the computation of certain partial recursive functions, outperforming the efficiency of classical computation, with a view towards a physically realizable beyond-Turing model of calculation.

\subsection*{Acknowledgments}
The authors were supported by the Spanish State Research Agency through projects PID2021-124440NB-I00, PID2023-146936NB-I00, RED2022-134301-T, and project PID2022-136795NB-I00 under MCIN/AEI/10.13039/501100011033 and  ERDF/EU. EM was supported by the Catalan Institution for Research and Advanced Studies through an ICREA Academia Prize 2021. All authors were partially supported by the project Computational, Dynamical and Geometrical Complexity in Fluid Dynamics (COMPLEXFLUIDS), Fundación BBVA 2021, and the bilateral AEI–DFG projects PCI2024-155042-2 and PCI2024-155062-2. This work also benefited from the Severo Ochoa and María de Maeztu Program for Centres and Units of Excellence in R\&D (CEX2020-001084-M and CEX2023-001347-S) of the Spanish State Research Agency.

\section{The Church-Turing Thesis}\label{sec:church-turing}

In this section, we shall review two of the most important models of computability that will play a major role in this work: Turing machines and partial recursive functions.

\subsection{Turing Machines}\label{sec:turing-machines}

This celebrated computability model captures the essence of a computer from a mathematical point of view. Due to its intuitive nature, it has become one of the most widely used models of computation. Formally, a (binary) \emph{Turing machine} is a tuple of data $M = (Q, q_0, Q_{\textrm{halt}}, \delta)$, where $Q$ is a finite set called the \emph{states} of $M$, $q_0 \in Q$ is the \emph{initial state}, $Q_{\textrm{halt}} \subsetneq Q$ is the set of \emph{halting states} and
$$
    \delta: Q \times \mathcal{A} \to Q \times \mathcal{A} \times \{-1,1\}
$$
is called the \emph{transition function}, with $\mathcal{A} = \{0,1, \square\}$ the alphabet of the tape. Here $\square$ must be seen as a ``blank'' character.

With this information, we can form a dynamical system as follows. Let us denote by $\Lambda \subseteq \mathcal{A}^{\ZZ}$ the set of finite two-sided sequences $t: \ZZ \to \mathcal{A} = \{0,1, \square\}$, i.e.,\ such that $t_n \neq \square$ only for finitely many $n \in \ZZ$. The elements $t \in \Lambda$ will be called the \emph{tape states}. The computation states are the set $\mathcal{S} = \{(t,q) \in \Lambda \times Q\}$, with $t \in \Lambda$ a tape state and $q \in Q$ the current state of the Turing machine. Then, we define the function
$$
    \Delta_M: \mathcal{S} \to \mathcal{S}
$$
as follows. Given a computation state $(t, q)$, the Turing machine computes $\delta(q, t_0) = (q', s, \epsilon)$ and sets $\Delta_M(t,q) = (t', q')$, where $t'$ is the tape state with $t'_n = t_{n+\epsilon}$ for $n \neq -\epsilon$ and $t'_{-\epsilon} = s$.

This function $\Delta_M$ represents the intuitive idea that $t$ is a two-sided tape with finitely many non-black symbols and $M$ has a reader-writer head placed on the $0$-th position of the tape. At each step, the head of $M$ reads the symbol under it and, according to the internal state $q$ of the machine, replaces it with the new symbol $s$ on the tape, shifts the head to the left ($\epsilon = -1$, or equivalently, the sequence to the right) or right ($\epsilon = 1$, or the sequence to the left) and the machine changes its internal state to $q'$. In this vein, given an initial tape $t$, the function $\Delta_M$ defines a dynamical system on $\mathcal{S}$ starting in $(t, q_0) \in \mathcal{S}$. If, at some point, the system reaches a state of the form $(t', q')$ with $q' \in Q_{\textrm{halt}}$, the dynamical system stops and we say that $M$ \emph{halts} with tape $t'$. Otherwise, we say that $M$ \emph{does not halt}.

\begin{remark}
    In principle, more general Turing machines can be considered if we allow for more symbols in the alphabet $\mathcal{A}$ of the tape rather than just $\mathcal{A} = \{0, 1, \square \}$. However, it can be shown that these machines with a larger alphabet are equivalent to a binary Turing machine, just by binary encoding each symbol in the tape. For this reason and their simplicity, we shall focus on binary machines in this work.
\end{remark}

\subsection{Partial Recursive Functions}

In this work, we shall specifically focus on a particular model of computability known as partial recursive functions, as introduced by Kleene in \cite{kleene1936general}. 

In this context, by a \emph{partial function} $f$ between sets $A$ and $B$ we mean a function $f: D_f \to B$ defined on a certain subset $D_f \subseteq A$, called the \emph{domain} of $f$. Equivalently, we can turn any partial function into a usual one by choosing a new element $\bullet$, setting $A^* = A \cup \{\bullet\}$ and $B^* = B \cup \{\bullet\}$, and defining the usual function $\tilde{f}: A^* \to B^*$ by $\tilde{f}(a) = f(a)$ if $a \in D_f$ and $\tilde{f}(a) = \bullet$ if $a \not\in D_f$. 
A partial function between $A$ and $B$ will be denoted by $f: A \dashrightarrow B$. Observe that two partial functions $f: A \dashrightarrow B$ and $g: B \dashrightarrow C$, with respective domains $D_f$ and $D_g$, can be composed as usual after restricting the domain of $f$ to $D_f \cap f^{-1}(D_g)$.

In this manner, the set of \emph{partial recursive functions} is the minimum set of partial functions $f: \NN^n \dashrightarrow \NN^m$ closed under composition and containing (1) the constant functions, (2) the projection functions, (3) the successor function, (4) primitive recursion functions and (5) minimization functions (see \cite{kleene1936general} for a detailed description). Furthermore, using the natural embedding $\iota: \NN^n \to \NN$ given by $\iota(x_1, \ldots, x_n) = p_1^{x_1} p_2^{x_2} \cdots p_n^{x_n}$, where $p_i$ denotes the $i$-th prime number, we can restrict ourselves to partial recursive functions of the form $f: \NN \dashrightarrow \NN$. Not every function is a partial recursive function, being an example the so-called busy beaver function \cite{rado1962non}.

The set of partial recursive functions represents a model of computation equivalent to the one given by Turing machines. Indeed, 
consider a binary Turing machine $M = (Q, q_0, Q_{\textrm{halt}}, \delta)$. From $M$, we can create the partial function $f_M: \NN \dashrightarrow \NN$ given by
$$
    f_M(x) = \left\{\begin{array}{ll}
        y & \textrm{if $M$ halts when $x$ is used as initial value and returns $y$},\\
        \bullet & \textrm{if $M$ does not halt when $x$ is used as initial value.}
    \end{array}\right. 
$$
In the description above, we encode $x$ in the initial tape through its binary form, with the least significant bit of $x$ on the left, and the head of the machine is placed on this least significant bit. The output is read in the same way: since only a finite number of cells in the tape are changed before $M$ halts, only a finite number of non-blank cells appear to the right of the position of the head after halting. These cells to the right of the head are seen as the encoding of a binary number $y$, which is interpreted as the output of $M$. In this manner, the function $f_M$ described above is said to be \emph{implemented} by the Turing machine $M$.

In this language, Kleene proved \cite{kleene1936general} that the function $f_M: \NN \dashrightarrow \NN$ implemented by a binary Turing machine $M$ is a partial recursive function. Reciprocally, every partial recursive function can be implemented via a Turing machine. Notice that the correspondence of the previous result is not bijective as different Turing machines can implement the same partial recursive function.

\subsection{Turing Machines as finite state machines}\label{sec:graph-automata}

There exists a geometric way of understanding Turing machines as graphs, typically referred to as \emph{finite state machines}. Suppose that we have a Turing machine $M = (Q, q_0, Q_{\textrm{halt}}, \delta)$. Then, we define the finite state machine $\cG_M$ associated to $M$ as the following labelled directed graph:
\begin{itemize}
    \item The vertices of $\cG_M$ are the states $Q$ of $M$.
    \item Every vertex $q \not\in Q_{\textrm{halt}}$ has three different outgoing edges, corresponding to the values of $\delta(q, 0) = (q', s', \epsilon')$, $\delta(q, 1) = (q'', s'', \epsilon'')$ and $\delta(q, \square) = (q''', s''', \epsilon''')$. In this way, we place an edge between $q$ and $q'$ labelled with the triple $(0, s', \epsilon')$, an edge between $q$ and $q''$ labelled with the triple $(1, s'', \epsilon'')$, and an edge between $q$ and $q'''$ labelled with  $(\square, s''', \epsilon''')$.
    \item The halting states of $Q_{\textrm{halt}}$ have no outgoing edges.
\end{itemize}
Notice that $\cG_M$ may contain loops if transitions of the form $\delta(q,s)=(q,s',\epsilon')$ exist.

The finite state machine $\cG_M$ associated to the Turing machine $M$ allows us to visualize the changes in the internal states of $M$ while computing. At the beginning of the computation, the system starts at the initial vertex $q_0$ and, at each step, the system transitions to a new vertex through the edge corresponding to the symbol read in the tape at that moment, writing the tape and moving left or right according to the label of the edge.

\begin{example}
Let us consider a Turing machine with three states $Q = \{q_0, q_1, q_2\}$, initial state $q_0$, halting states $Q_{\textrm{halt}} = \{q_2\}$ and transition function given by
$$
    \begin{array}{lclclclclcl}
        \delta(q_0, 0) & = & (q_0, 0, +1), & \quad & \delta(q_0, 1) & = & (q_1, 0, +1), & \quad & \delta(q_0, \square) & = & (q_2, 1, -1), \\
        \delta(q_1, 0) & = & (q_1, 0, +1), & \quad & \delta(q_1, 1) & = & (q_0, 0, +1), & \quad & \delta(q_1, \square) & = & (q_2, 0, -1). \\
    \end{array}
$$
This Turing machine is designed to read a string of 0's and 1's and to return $1$ if it finds an even number of 1's and $0$ otherwise. According to the previous description, the associated finite state machine is the one depicted in \cref{fig-example-graph}. The halting states are highlighted with a double border, and the initial state is marked with an incoming edge from no node.

\begin{figure}[h!]
    \centering
    \begin {tikzpicture}[-latex ,auto ,node distance =3 cm and 3cm ,on grid ,
semithick ,
state/.style ={ circle ,top color =white , bottom color = processblue!20 ,
draw,processblue , text=blue , minimum width =1 cm}]
\node[state,accepting] (C)
{$q_2$};
\node[state,initial by arrow, initial text={}, initial above] (A) [above left=of C] {$q_0$};
\node[state] (B) [above right =of C] {$q_1$};
\path (A) edge  node[left] {$(\square,1,-1)\;$} (C);
\path (A) edge [bend left =25] node[above] {$(1,0,+1)$} (B);
\path (B) edge [bend left =15] node[below =0.15 cm] {$(1,0,+1)$} (A);
\path (B) edge node[right] {$\;(\square,0,-1)$} (C);
\path (A) edge [loop left] node[left] {$(0,0,+1)$} (A);
\path (B) edge [loop right] node[right] {$(0,0,-1)$} (B);
\end{tikzpicture}
    \caption{Finite state machine for the parity detector Turing Machine.}
    \label{fig-example-graph}
\end{figure}
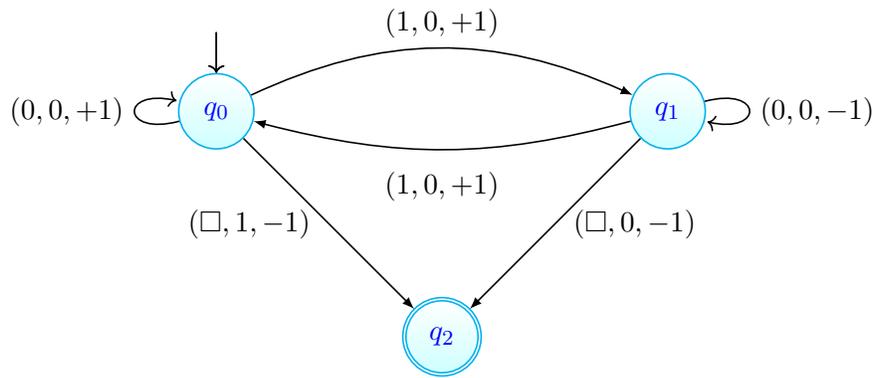
\end{example}

\section{Dynamical computation}\label{sec:TKFT}

In this section, we introduce a new tool to understand computability from a geometric point of view. As we shall show, this provides a framework to reinterpret several computability questions, such as complexity and their limitation. However, to define these concepts properly, we will need some geometric and combinatorial background.

\subsection{Dynamical bordisms}

Throughout this paper, an $n$-dimensional \emph{manifold} with ($2$-step) corners will be a second countable Hausdorff topological space $W$ with local charts $\varphi: U \subseteq M \to V \subseteq \HH_2^n$ into open sets of $\HH_2^n = \{(x_1, \ldots, x_n) \in\RR^n\mid x_{n-1},x_n \geq 0\}$ and smooth transition functions. Given a point $p \in M$ with image $\varphi(p) = (x_1, \ldots, x_n)$, if $x_{n-1} > 0$ and $x_n > 0$, we say that it is an \emph{interior point}; if $x_{n-1} = 0$ and $x_{n} > 0$, or $x_{n-1} > 0$ and $x_{n} = 0$, we call it a \emph{frontier point}; and if $x_{n-1} = x_n = 0$, we refer to it as a \emph{corner point}. The set of frontier points will be denoted by $\partial_1 W$ and the set of corner points by $\partial_2 W$. Additionally, the set $\partial W = \partial_1 W \cup \partial_2 W$ will be called the \emph{boundary} of $W$. Notice that every component of $\partial_1 W$ is a usual manifold with boundary, and $\partial W$ is the result of non-smoothly gluing these components along their boundaries.

In this setting, given two manifolds $M_1$ and $M_2$ with boundary, a \emph{bordism} between them will be a manifold with corners $W$ such that we can decompose its boundary as $\partial W = M_1 \sqcup M_2 \sqcup \partial'W$, where $\partial'W \subseteq \partial_1W$ is an open manifold without boundary. Additionally, a smooth vector field $X$ on the bordism $W$ is said to be \emph{tame} if $X$ is transverse to $M_1$ and $M_2$ and tangent to $\partial' W$. Furthermore, notice that, for $i = 1, 2$, we have a splitting of the tangent bundle $TW|_{M_i} = TM_i \oplus \RR$ as a vector bundle over $M_i$, and thus the tameness condition implies that $X|_{M_i}$ is a nowhere vanishing vector field in the ``normal'' direction.

\begin{definition}
    A \emph{dynamical germ} is a pair $(M, X)$ of an $n$-dimensional compact manifold with boundary together with a non-zero smooth vector field $X$ on $TM \oplus \RR$ with $X = (V, s)$ and $s\equiv s(p) < 0$ for all $p\in M$. 
    Given dynamical germs $(M_1, X_1)$ and $(M_2, X_2)$, a \emph{dynamical bordism} between them is a pair $(W, X)$ where $W$ is an $(n+1)$-dimensional compact bordism with corners between $M_1$ and $M_2$ and $X$ is a tame vector field on it such that $s|_{M_1} = -s_1$ and $s|_{M_2} = s_2$, where we are decomposing $X_i=(V_i,s_i)$ as before.
\end{definition}

For a dynamical germ $(M, X)$, we should think about the vector field $X$ as pointing ``outwards'' when $M$ is seen as a boundary of a bordism. For this reason, that $s|_{M_1} = -s_1$ for a bordism $(W, X)$ between $(M_1, X_1)$ and $(M_2, X_2)$ means that $X$ points ``inwards'' at $M_1$, while $s|_{M_2} = s_2$ means that it points ``outwards'' at $M_2$.

\subsection{Cantor-based encodings}\label{sec:cantor-encoding}

{
The standard ternary Cantor set is the set of real numbers $x\in \mathbb{R}$ of the form $x = \sum_{i=1}^\infty \epsilon_i 3^{-i}$, with $\epsilon_i = 0$ or $2$. This Cantor set can be used to encode natural numbers (through their binary expansion) in the real line or, even more interestingly, paths in a binary tree ($\epsilon_i = 0$ means that we turn left in the $i$-th branch, and $\epsilon_i = 2$ that we turn right).

Using these ideas, we can define a useful encoding of finite binary sequences with symbols from the alphabet $\mathcal{A} = \{0,1, \square\}$ into the closed interval $I = [-1,1] \subseteq \RR$. For this purpose, let $t = \{t_n\}$ be a finite sequence with $t_n \in \mathcal{A}$ for $n \geq 0$ and $t_m = \square$ for $m$ large enough. We encode $t$ in $I$ as the point
$$
    x_t = \frac{2}{3}\sum_{i=0}^\ell \sigma(t_i)3^{-i},
$$
where $\sigma: \mathcal{A} \to \{-1,0,1\}$ is the assignment $\sigma(0) = -1$, $\sigma(1) = 1$ and $\sigma(\square) = 0$. In other words, we see $t$ as a finite path in a binary tree where $t_n = 0$ means turning left, $t_n = 1$ means turning right and $t_n = \square$ indicates no turn.
If we identify the set of finite sequences with the set of natural numbers according to their binary expansion (ending with a tail of blank symbols $\square$), the construction above defines an embedding of $\NN$ in the interval $I$ that we shall denote by $\kappa: \NN \hookrightarrow I$.

This construction is particularly useful to encode finite two-sided binary sequences, that is, sequences of the form $t = \{t_n\}$ with $n \in \ZZ$ and $t_m = \square$ for $|m|$ large enough.  These sequences form a set denoted by $\Lambda$. Indeed, if we split $t$ into the two one-sided sequences $t^+$ and $t^-$ of non-positive and negative terms, respectively, we define the embedding $\kappa: \Lambda \hookrightarrow I^2$, $t \mapsto (\kappa(t^-), \kappa(t^+))$. In the same way, we obtain embeddings of $\NN^n$ in the cube $I^n$ for any $n\geq 1$, also denoted by $\kappa: \NN^n \hookrightarrow I^n$. %
}

Using this notion, we can enlarge our notion of dynamical bordisms to encode embeddings of intervals.

\begin{definition}
    A \emph{marked dynamical germ} is a triple $(M, X, \bm{\iota})$, where $(M, X)$ is a dynamical germ and $\bm{\iota} = \{\iota_1, \ldots, \iota_m\}$ is a collection of embeddings $\iota_j: I^n \hookrightarrow M_j$ for $1 \leq j \leq m$ of the $n$-dimensional interval, where $M = \sqcup_{j=1}^m M_j$ is the decomposition of $M$ into its connected components. 
\end{definition}

\begin{example}
The tuple $(\mathbb{D}, X, \iota)$ is a marked dynamical germ, where $\mathbb{D}$ is the standard $n$-dimensional disk of radius $2$, $\iota: I^n \hookrightarrow\mathbb{D}^n$ the usual inclusion and $X_p = (0, -1)$ for all $p \in \mathbb{D}^n$ is the normal vector field. We shall refer to this marked dynamical germ as the \emph{standard disc}. 
\end{example}

An important class of dynamical bordisms arises by considering topologically trivial bordisms with dynamics governed by a diffeomorphism. Suppose that $(M, X, \iota)$ is a connected marked dynamical germ such that the image of the embedding $\iota: I^n \hookrightarrow M$ lies in a chart of $M$ and agrees with the standard cube $I^n$ in the chart. In that case, a diffeomorphism $\varphi: M \to M$ is said to be \emph{computable} if $\varphi(\iota(I^n)) = \iota(I^n)$ and $\varphi|_{\NN^n}: \NN^n \subseteq I^n \to \NN^n \subseteq I^n$ is a partial recursive function. If $\varphi$ is moreover isotopic to the identity, then such isotopy generates a nowhere vanishing vector field $X_\varphi$ on $M \times [0,1]$ so that the flow of $X_\varphi$ starting at $(x, 0)$ ends at $(\varphi(x), 1)$. Endowed with this vector field, $M \times [0,1]$ becomes a dynamical bordism of marked dynamical germs whose flow represents~$\varphi$.

{
\begin{definition}
    \label{ex:basic-bordism}
     Let $\varphi: M \to M$ be a null-homotopic computable diffeomorphism. The dynamical bordism $(M \times [0,1], X_\varphi)$ is called a \emph{basic bordism}.
\end{definition}
}

{
Using these basic dynamical bordisms, we can construct more complex dynamical bordisms by gluing them. This generates an important class of dynamical bordisms that we refer to as clean. 

\begin{definition}\label{defn:clean}
A dynamical bordism $(W, X)$ is said to be \emph{clean} if there exists a CW-subcomplex $\widetilde{W} \subseteq W$ such that $\widetilde{W}$ is a deformation retract of $W$ and $\widetilde{W}$ can be decomposed as the gluing of finitely many basic dynamical bordisms of marked dynamical germs glued along subcomplexes of $\widetilde{W}$.
\end{definition}
}

\subsection{Partial functions from dynamical bordisms}\label{sec:pf-bordisms}

Let $(W, X)$ be a dynamical bordism between marked dynamical germs $(M, X, \bm{\iota})$ and $(M', X', \bm{\iota}')$. Suppose that $M$ has $m$ connected components and $M'$ has $m'$ components. In this section, we will show that we can naturally associate to the bordism $(W, X)$ a partial function $Z_0(W, X): \NN^{\sqcup m}  \dashrightarrow \NN^{\sqcup m'}$, where $\NN^{\sqcup s} = \NN \times \{1, \ldots, s\}$ is the disjoint union of $s$ copies of $\NN$.

To do so, for simplicity, let us first suppose that both $M$ and $M'$ are non-empty and connected, we only have two embeddings $\iota: I^n \hookrightarrow M$ and $\iota': I^n \hookrightarrow M'$. We define the auxiliary partial map
$$
    \psi_{(W, X)}: \NN^n \dashrightarrow \NN^n
$$
as follows. Pick $c \in \NN^n$ and let $x = \iota(\kappa(c)) \in M$, where $\kappa: \NN^n \hookrightarrow I^n$ is the embedding of Section \ref{sec:cantor-encoding}. Consider the flow of the vector field $X$ on $W$ starting at $x$. If the flow hits $M' \subseteq W$, let $x'$ be such an intersection point. Notice that, since $X$ points outwards at $M'$, there exists at most one such intersection point. Then, we set
\begin{equation}\label{eq:map-cantor}
    \psi_{(W, X)}(c) = \left\{\begin{array}{ll}
        (\iota' \circ \kappa)^{-1}(x') & \textrm{if $x' \in \textrm{Img}(\iota' \circ \kappa)$}, \\
        \bullet & \textrm{otherwise}. 
    \end{array}\right.
\end{equation}

\begin{remark}
We may have $\psi_{(W, X)}(c) = \bullet$ for two reasons: Either the flow starting at $x$ does not hit the outgoing boundary $M'$ or it hits $M'$ but not in the image of the Cantor set.
as represented in Figure \ref{fig:dyn-morph}.
\begin{figure}[h]
    \centering
    \includegraphics[width=0.45\linewidth]{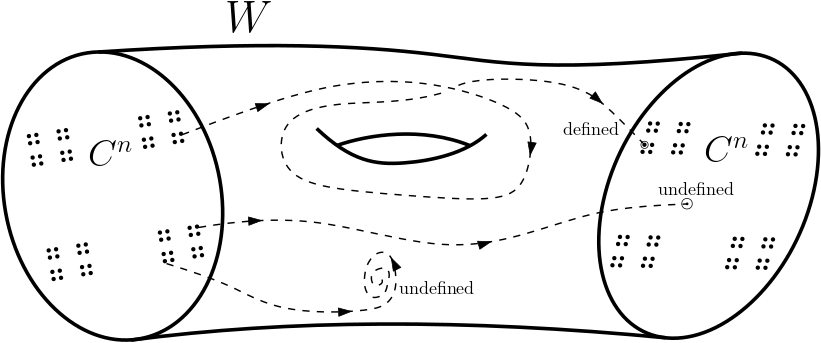}
    \caption{Different possibilities occurring in the definition of $\psi_{(W, X)}: C^n \dashrightarrow C^n$.}
    \label{fig:dyn-morph}
\end{figure}
\end{remark}

With this information at hand, we define the partial function
$$
    Z_0(W, X): \NN  \dashrightarrow \NN
$$
as follows. Let $n \in \NN$ and consider the ``line'' $\{n\} \times \kappa(\NN^{n-1}) \subseteq I^n$. If $\psi_{(W, X)}\left(\{n\} \times \kappa(\NN^{n-1})\right) \subseteq \{n'\} \times \kappa(\NN^{n-1})$ for some $n' \in \NN$, then we set $Z_0(W, X)(n) = n'$, and otherwise $Z_0(W, X)(n) = \bullet$. Explicitly, $Z_0(W, X)(n) = n'$ if and only if, for all $y$, $\psi_{(W, X)}(n, y) = (n', y')$ for some $y' \in \kappa(\NN^{n-1})$.

\begin{remark}
    The clumsy condition that the line $\{n\} \times \kappa(\NN^{n-1})$ must be mapped to a line $\{n'\} \times \kappa(\NN^{n-1})$ should be understood as the fact that the last $n-1$ components of the initial point are just ``auxiliary data'' that do not affect the final result of the calculation. The need for these auxiliary data is very clear geometrically, since the reaching function of a smooth flow is injective but we want to represent also non-injective functions. To avoid this problem, we use the vertical direction to store auxiliary data that allows us to separate points.
    
\end{remark}

The general case of a dynamical bordism $(W, X): (M, X, \bm{\iota}) \to (M', X', \bm{\iota}')$ between non-connected objects can be worked similarly component-wise 
to get a partial map
\begin{equation*}
    \psi_{(W, X)} = \sqcup_j \psi_{(W, X),j}: (\NN^n)^{\sqcup m} \dashrightarrow (\NN^n)^{\sqcup m'},
\end{equation*}
where $m$ is the number of components of $M$ and $m'$ the one of $M'$.
As in the previous case, this induces a partial function $Z_0(W, X): \NN^{\sqcup m}  \dashrightarrow \NN^{\sqcup m'}$. We have thus proven the following result.

\begin{proposition}
    Every dynamical bordism $W$ between marked dynamical germs $M$ and $M'$ with $m$ and $m'$ connected components, respectively, can be naturally associated with a partial function
    $$
        Z_0(W): \NN^{\sqcup m}  \dashrightarrow \NN^{\sqcup m'}.
    $$
    These partial functions satisfy that $Z_0(W \cup_{M'} W)$ is an extension of $Z_0(W) \circ Z_0(W')$ for any dynamical bordisms $W$ and $W'$ with common boundary $M'$.
\end{proposition}

At this point, an important question arises: What partial functions are representable by a clean dynamical bordism? In other words, given a partial function $f: \NN \dashrightarrow \NN$, does there exist a clean dynamical bordism $W_f$ such that $Z_0(W_f) = f$?
In this direction, the following theorem establishes the first part of Theorem~\ref{thm:main-theorem-intro}, whose proof is provided in Section \ref{sec:prf-quantizable}.

\begin{theorem}\label{thm:main}
    Partial recursive functions are representable by volume-preserving clean dynamical bordisms of the standard disc.
\end{theorem}

{\begin{remark}
    In general, the topology of the bordism representing a partial recursive function $f: \NN \to \NN$ will not be trivial, capturing the complexity of the function. In particular, the function $f$ can be represented by a basic morphism if and only if $f$ extends to a diffeomorphism of the disc, when seen $\NN \hookrightarrow \DD$ via $\kappa$ (recall that $\DD$ is the standard disk of radius $2$). However, not all the partial recursive functions can be extended, such as the function that reverses the binary expansion of a natural number, i.e., $f\left(\sum_{i=0}^n a_i2^i\right) = \sum_{i=0}^n a_{n-i}2^i$. Indeed, the points in $\mathbb D$ associated to the sequences $x_n = 2^n$ and $x'_n = 2^n + 2^{n-1}$ converge to the same limit $(-1, 0) \in \DD$, but $f(x_n) = 1$ and $f(x_n') = 3$ for all $n$.
\end{remark} }

{

Recall that, by definition, the vector field of a dynamical bordism is transverse to the incoming and outgoing boundaries of the bordism. This condition is crucial in Theorem \ref{thm:main}. If we drop it, it is easy to construct a vector field on a topologically trivial bordism representing any function, just by extending the flowlines on the Cantor set by zero.
Indeed, if $A \subseteq M$ is any discrete subset of a compact manifold $M$ (possibly with boundary), given any function $f: A \to A$, we can easily construct a degenerate non-transverse dynamical bordism representing it. To do so, in $M \times [0,1]$, choose pairwise disjoint paths $\gamma_a$ connecting $(a, 0)$ with $(f(a), 1)$ for all $a \in A$. Since $A$ is discrete, these paths $\gamma_a$ can be thickened into pairwise disjoint closed tubes $T_a \subseteq M \times [0,1]$. We endow each of these tubes $T_a$ with a flow-box vector field $X_a$ in such a way that the flow starting at $(a,0)$ hits $(f(a), 1)$ and such that $X_a$ vanishes at the `vertical' boundary of $T_a$. Extending these vector fields $X_a$ by zero outside $T_a$, we get a vector field $X$ on $M \times [0,1]$ representing the function $f$ on $A$. However, this vector field $X$ is clearly non-transverse to the boundaries $M \times \{0,1\}$ since it vanishes on an open set.  

In the same spirit, notice that the condition that the dynamical bordism $(W, X)$ representing a partial recursive function $f$ is clean is also of central relevance here. Intuitively it forces that the flow of $X$ captures topologically the operation of $f$. Indeed, since $(W,X)$ is constructed by gluing basic bordisms induced by a computable diffeomorphism, each of these basic tubes can only represent a simple intermediate calculation performed during the computation of $f$.
}

\begin{remark}
The construction of Theorem \ref{thm:main-theorem-intro} can be improved to allow an encoding that is robust under non-uniform perturbations. To be precise, instead of encoding a finite sequence $t = \{t_i\}$ by a point in the square $I$, it can be encoded by a slightly bigger open set $U_{t} \subseteq I$ given by
$$
    U_t = \left\{\left.\frac{2}{3}\sum_{i=0}^\ell \sigma(t_i)3^{-i} + y \in I \;\right|\; \frac{1+\varepsilon}{3^{\ell +2}}< y < \frac{2-\varepsilon}{3^{\ell +2}}\right\},
$$
for $\varepsilon > 0$ fixed and small. In other words, $U_t$ is the interval of length $(1-2\varepsilon)3^{-\ell - 2}$ centered in $x_t$, where $x_t \in I$ is the point defined in Section \ref{sec:cantor-encoding}. This induces a collection of open sets $\mathcal{U}_{\NN} = \{U_t\}_{t \in \NN}$ of $I$. By considering the collection $\mathcal{U}_{\Lambda} = \{U_{t^+} \times U_{t^-}\}$ for $t \in \Lambda$ a two-sided finite sequence with non-negative part $t^+$ and negative part $t^-$, we can also form a collection of open sets of $I^2$ encoding $\Lambda$. The whole construction of this paper can be carried out verbatim to work with this encoding in terms of open sets.
\end{remark}

{
We finish this section by proving that the converse of Theorem \ref{thm:main} also holds, which completes the proof of Theorem~\ref{thm:main-theorem-intro}. Explicitly, given any clean dynamical bordism $(W, X)$, the flow of $X$ defines a function $Z_0(W, X): (\NN^n)^{\sqcup m} \dashrightarrow (\NN^n)^{\sqcup m'}$ that is computable. 

\begin{theorem}\label{thm:clean-kleen}
    The reaching function of a clean dynamical bordism is a partial recursive function.
\end{theorem}

\begin{proof}
    We will prove it by induction on the number $b$ of basic bordisms glued to form the clean dynamical bordism. In the base case $b = 1$, the clean bordism is a basic bordism and the reaching function is computable by hypothesis.

    Now, suppose that the result is true for all clean bordisms formed by gluing $b\geq 1$ basic bordisms, and let us prove it for $b + 1$. Let $W$ be a clean bordism whose underlying CW-complex $\widetilde{W}$ is formed by gluing $b + 1$ basic bordisms and has reaching function $Z_0(\widetilde{W}): \NN^{\sqcup m} \to \NN^{\sqcup m'}$. Choose one of the basic bordisms $M_\varphi = (M \times [0,1], X_\varphi)$ in $\widetilde{W}$ and denote by $\widetilde{W}'$ the result of removing $M_\varphi$ from $\widetilde{W}$. We have that $\widetilde{W}'$ is composed of $b$ basic bordisms, so by induction hypothesis, the reaching function $Z_0(\widetilde{W}')$ is computable. 
    
    First, suppose that the input and output boundaries of $M_\varphi$ are not boundaries of the bordism $\widetilde{W}$. Hence, removing $M_\varphi$ has the effect of adding a new input and output boundaries in $\widetilde{W}'$, so that $Z_0(\widetilde{W}'): \NN^{\sqcup (m+1)} \to \NN^{\sqcup (m'+1)}$ and $Z_0(M_\varphi): \NN \to \NN$. For simplicity, let us assume that the newly created input and output boundaries correspond to the $(m+1)$-th copy of $\NN^{\sqcup (m+1)}$ and the $(m'+1)$-th copy of $\NN^{\sqcup (m'+1)}$, respectively. Let us denote by $(n, i)$ the point $n \in \NN$ in the $i$-th component of $\NN^{\sqcup m}$. In that situation, the function $Z_0(\widetilde{W})$ can be described in terms of $Z_0(\widetilde{W}')$ and $Z_0(M_\varphi)$ as follows: $Z_0(\widetilde{W})(n,i) =
            (n', j)$ if  $Z_0(\widetilde{W}')(n,i) = (n',j)$ with $j \leq m'$, and
            $Z_0(\widetilde{W})(n,i) = Z_0(\widetilde{W}')(Z_0(M_\varphi)(n'), m+1)$ if $Z_0(\widetilde{W}')(n,i) = (n',m'+1)$.

    Using the inductive construction of partial recursive functions, it is easy to check that this function can be written as a recursion on $Z_0(\widetilde{W}')$ and thus it is a partial recursive function provided that $Z_0(\widetilde{W}')$ is so. The case in which $M_\varphi$ connects with an output boundary of $\widetilde{W}$ is completely analogous by adjusting the number of input and output components of the boundary of $\widetilde{W}'$, and results in a simpler expression for $Z_0(\widetilde{W})$. This proves the inductive step and therefore the theorem follows.
\end{proof}

\begin{corollary}
    Clean dynamical bordisms is a model of computation equivalent to Turing machines.
\end{corollary}
}

\section{Proof of the main theorem}\label{sec:prf-quantizable}

In this section, we shall prove Theorem \ref{thm:main} by showing that any finite state machine, as described in Section \ref{sec:graph-automata}, can be represented through a dynamical bordism. For simplicity, we shall focus on the case $n = 2$ of surfaces, but the constructions can be straightforwardly extended to the higher dimensional case. 

\subsection{Reversible Turing machines}\label{sec:reversible-Turing}

A very desirable feature of a Turing machine is the ability to reverse their computation. To be precise, consider a Turing machine $M$ that induces a dynamical system
$$
    \Delta_M: \mathcal{S} \to \mathcal{S}
$$
on the space $\mathcal{S}$ of computational states. We say that $M$ is \emph{reversible} if the function $\Delta_M$ is injective. Notice that, in particular, this implies that if $\delta: Q \times \mathcal{A} \to Q \times \mathcal{A} \times \{\pm 1\}$ is the transition function of $M$, then its first two components $\delta': Q \times \mathcal{A} \to Q \times \mathcal{A}$ form an injective function. Reversible computation is an active area of research that has been deeply studied in the literature, see for instance \cite{barbieri2016group,morita2017theory}.

In the seminal paper \cite{bennett1973logical}, Bennett showed that any computation can be done by means of a reversible Turing machine, as stated in the following result.

\begin{theorem}{\cite{bennett1973logical}}
For every Turing machine, there exists an equivalent reversible Turing machine.
\end{theorem}

To perform the reversible calculation, the new Turing machine uses three tapes, meaning that the Turing machine has access to three separate tapes with three read-write heads on each of them so that at every step only one of the heads performs an operation (see \cite{sipser1996introduction} for more information about multi-tape machines). Additionally, the construction developed by Bennett has the property that, when the computation has finished, the output of the reversible Turing machine is just the output of the calculation in one tape, whereas another tape contains the initial input tape and the third tape is completely blank. In particular, no auxiliary information about the history of transitions of the Turing machine remains in the tape when the calculation concludes.

It is worth mentioning that the use of several tapes for the calculation is not a limiting factor, and the same calculation can be done with a single-tape machine. This fact is well known, but we include a proof here for the sake of completeness.

\begin{proposition}\label{prop:reversible-tm}
For every multi-tape reversible Turing machine, there exists a single-tape reversible Turing machine computing the same function.
\end{proposition}

\begin{proof}

First, observe that more symbols can be added to the tape alphabet $\mathcal{A}$ of a reversible Turing machine without affecting the reversibility of the computation. To do so, we encode the symbols of $\mathcal{A}$ in binary in the tape by using $b$ bits, and substitute each transition by a chain of transitions reading successively the $b$ symbols of the tape, and replacing them back before shifting. Observe that the amount $b$ of spots needed to encode these symbols is fixed beforehand since $\mathcal{A}$ is finite. The computation is still reversible since we are only adding states with no extra loops.

Additionally, without affecting the reversibility, it is possible to perform the operation that, given a tape, shifts all the subtape to the right of the reading head one position to the right, leaving the subtape to the left unchanged and adding a special symbol in the newly created space. To do so, successively replace every symbol in the right substring by the one immediately to its left, recording the deleted symbol in an auxiliary state of the machine. Thanks to the special symbol marking the starting point of the shift, the computation is clearly reversible. Iterating this machine, we can reversibly shift the right substring $N > 0$ positions to the right and, analogously, we can perform the same operation to the left.

Using these operations, we can simulate $n$ tapes with a reversible single-tape Turing machine as follows. We encode the $n$ tapes in a single tape side-by-side, juxtaposing the symbols and separating them with a special symbol working as delimiter. The current position of the read-write head on each of the tapes is also marked by special symbols. The amount of space devoted to each tape is controlled by a fixed natural number $N > 0$ (typically, very large), so that each subtape is allocated in a memory of $N$ bits. Each time one of the substrings runs out of space and needs more than $N$ bits, a new block of $N$ bits is allocated to its left by applying the shifting machine described above. The process is illustrated in Figure \ref{fig:shift-memory}.

\begin{figure}[h]
    \centering
    \includegraphics[width=0.5\linewidth]{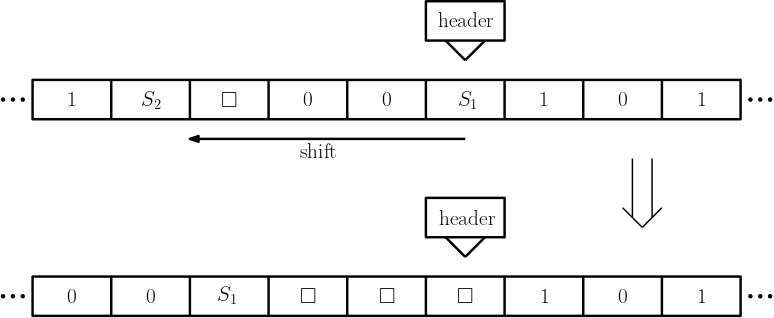}
    \caption{Addition of new memory for the space allocated for a tape. In the picture, every tape has reserved blocks of size $N = 3$, $S_1$ and $S_2$ denote the separation symbols between tapes, and $\square$ is the blank symbol.}
    \label{fig:shift-memory}
\end{figure}

In this manner, the single-tape Turing machine operates on each tape as usual, and every time we want to change the computation to a different tape, we shift (left or right, depending on the position of the target tape) until we find the special symbol pointing the position of the head in that tape. Observe that the amount of auxiliary symbols needed to perform this simulation only depends on the number of tapes $n$, and thus can be fixed beforehand.
\end{proof}

\subsection{The basic bordism}\label{sec:representation-basic}
Before proceeding with the proof of the main result, let us show that we can represent the basic read-write-shift operation as a bordism using elementary operations. For simplicity, throughout this section we will consider the case in which we have only two symbols in our alphabet, $0$ and $1$, but it is straightforward to adapt it to work with more symbols. 

Recall that we denote by $\Lambda$ the set of finite two-sided sequences. For $\epsilon = \pm 1$, we consider the shifting automorphism of $\Lambda$,
$$
    s_\epsilon: \Lambda \to \Lambda,
$$
given by $[s_\epsilon(t)]_n = t_{n+\epsilon}$ for all $n \in \ZZ$. In other words, for $\epsilon = 1$, $s_1$ shifts the sequence to the left (or the origin to the right); whereas for $\epsilon = -1$, $s_{-1}$ shifts the sequence to the right (or the origin to the left). It is well known that $s_\epsilon$ can be extended to an area-preserving diffeomorphism of the disk $\mathbb{D}$, which is equivalent to the celebrated horseshoe map (recall that $\DD$ denotes the standard disk of radius $2$). We include a proof for the sake of completeness. 

\begin{lemma}\label{lem:basic-diffeo}
    For any $\epsilon \in \{\pm 1\}$, there exists a computable area-preserving smooth diffeomorphism
    $$
        \varphi_{\epsilon}: \mathbb{D} \to \mathbb{D}
    $$
    that is the identity on the boundary of $\mathbb{D}$ and such that $\varphi_{\epsilon}|_{\Lambda} = s_{\epsilon}$.
\end{lemma}

\begin{proof}
Let $I_\varepsilon = [-\varepsilon, 1 + \varepsilon]^2$  for $\varepsilon > 0$ small. Divide the square $I_\varepsilon$ into four regions, $X$, $Y$, $Z$ and $T$, as shown in Figure \ref{fig:diffeo-psi}. The points of $\Lambda$ inside each region correspond to those two-sided sequences of the form
$$
    \begin{array}{cc}
        X \cap \Lambda = \left\{ \ldots \underset{0}{0}0 \ldots\right\}, & Y \cap \Lambda = \left\{ \ldots \underset{0}{0}1 \ldots\right\}, \\
        Z \cap \Lambda = \left\{ \ldots \underset{0}{1}0 \ldots\right\}, & T \cap \Lambda = \left\{ \ldots \underset{0}{1}1 \ldots\right\}. \\
    \end{array}
$$
Here, the subscript ``$0$'' denotes the position of the $0$-th digit in the two-sided sequence. Now, through a linear transformation, we map these regions into the new regions $X'$, $Y'$, $Z'$ and $T'$, as displayed in Figure \ref{fig:diffeo-psi}. Notice that the points of the Cantor set in these regions correspond to sequences of the form
$$
    \begin{array}{cc}
        X' \cap \Lambda = \left\{ \ldots 0\underset{0}{0} \ldots\right\}, & Y' \cap \Lambda = \left\{ \ldots 0\underset{0}{1} \ldots\right\}, \\
        Z' \cap \Lambda = \left\{ \ldots 1\underset{0}{0} \ldots\right\}, & T' \cap \Lambda = \left\{ \ldots 1\underset{0}{1} \ldots\right\}. \\
    \end{array}
$$
Since the new regions are pairwise disjoint, it is easy to extend the aforementioned linear transformation to an area-preserving diffeomorphism of the whole disk $\varphi_{+}: \mathbb{D} \to \mathbb{D}$ so that it is the identity on the boundary, and $\varphi_+|_{\Lambda} = s_+$ (see~\cite[Proposition 5.1]{CMPP} for details). Taking $\varphi_- = \varphi_+^{-1}$, we get and area-preserving diffeomorphism with $\varphi_-|_{\Lambda} = s_-$.

\begin{figure}[h]
    \centering
    \includegraphics[width=0.6\linewidth]{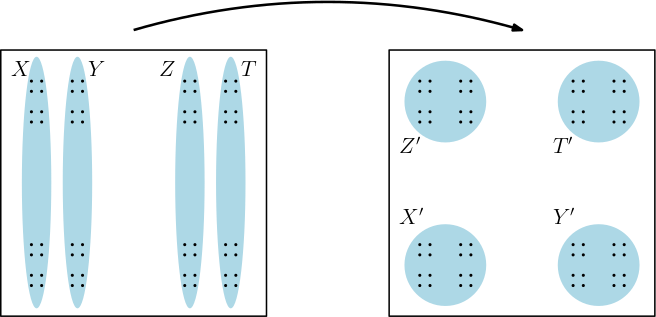}
    \caption{Operation of the diffeomorphism $\varphi_+: \mathbb{D} \to \mathbb{D}$.}
    \label{fig:diffeo-psi}
\end{figure}
\end{proof}

Now, observe that any diffeomorphism of the disk that is the identity on the boundary is isotopic to the identity with a boundary-fixing isotopy. Hence, using the construction in Definition \ref{ex:basic-bordism}, we get the following result. In the statement, by trivial we mean that the vector field $X$ is of the form $\partial_z$ in a neighborhood of  $\partial(\mathbb D\times [0,1])$, with $z\in[0,1]$.

\begin{corollary}\label{cor:basic-bordism}
    For any $\epsilon \in \{\pm 1\}$, there exists a basic bordism
    $$
        W_{\epsilon} = (\mathbb{D} \times [0,1], X_{\epsilon}): \mathbb{D} \to \mathbb{D}
    $$
    such that $X$ is volume-preserving, trivial in a neighborhood of $\partial(\mathbb{D} \times [0,1])$, and the map $\psi_{W_{\epsilon}}: \Lambda \to \Lambda$ of (\ref{eq:map-cantor}) associated to it is $\psi_{W_{\epsilon}} = s_{\epsilon}$.
\end{corollary}

However, a computation step of a Turing machine is not only a shift, but also a read-write operation. We can represent this operation by gluing $W_\epsilon$ to suitable parts of disks. To be precise, suppose that we want to execute a transition $(\alpha, \beta, \epsilon)$, i.e., when the symbol $\alpha \in \mathcal{A}$ of the alphabet is read, the header substitutes it by $\beta \in \mathcal{A}$ and then shifts the tape in the $\epsilon = \pm 1$ direction.

For any symbol $s \in \mathcal{A}$, let us consider disjoint diffeomorphic open sets $\mathbb{D}_{0,s}$ of the disk $\mathbb{D}$ containing the vertical strip of all the points of $\Lambda$ of sequences $t$ with $t_0 = s$, as in Figure \ref{fig:open-sets}. Using the shifting diffeomorphism $\varphi_+: \mathbb{D} \to \mathbb{D}$ of Lemma \ref{lem:basic-diffeo}, we also set $\mathbb{D}_{n,s} = \varphi_-^n(\mathbb{D}_{0,s})$ for any $n \in \mathbb{Z}$. In this manner, the open set $\mathbb{D}_{n,s}$ contains the sequences $t \in \Lambda$ with $t_n = s$, see Figure \ref{fig:open-sets}.

\begin{figure}[h]
    \centering
    \includegraphics[width=0.7\linewidth]{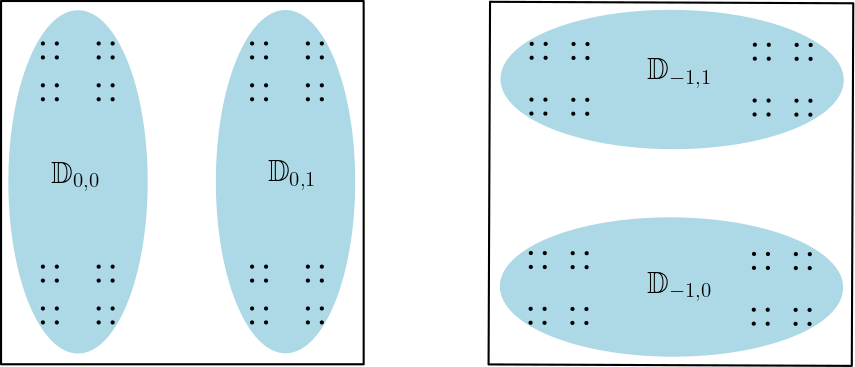}
    \caption{Open sets $\mathbb{D}_{n,s}$.}
    \label{fig:open-sets}
\end{figure}

Now, we are going to paste $W_\epsilon$ to an incoming disk $\mathbb{D}$ and to an outgoing disk $\mathbb{D}'$ to represent the read-write-shift operation $(\alpha, \beta, \epsilon)$. For this purpose, let us denote the incoming boundary of $W_\epsilon$ by $\mathbb{D}(W_\epsilon)$ and the outgoing boundary by $\mathbb{D}'(W_\epsilon)$. The partition defined above determines open sets $\mathbb{D}(W_\epsilon)_{0, \beta}$ and $\mathbb{D}'(W_\epsilon)_{-\epsilon, \beta}$ of $\mathbb{D}(W_\epsilon)$ and $\mathbb{D}'(W_\epsilon)$, respectively. Then, we glue a small cylinder $\mathbb{D}(W_\epsilon)_{0, \beta} \times [0,1]$ to $\mathbb{D}(W_\epsilon)_{0, \beta}$ and analogously a cylinder $\mathbb{D}'(W_\epsilon)_{-\epsilon, \beta} \times [0,1]$ to $\mathbb{D}'(W_\epsilon)_{-\epsilon, \beta}$. The resulting space will be denoted $\widehat{W}_\epsilon$ and is depicted in Figure \ref{fig:gluing-short}.

\begin{figure}[h]
    \centering
    \includegraphics[width=0.7\linewidth]{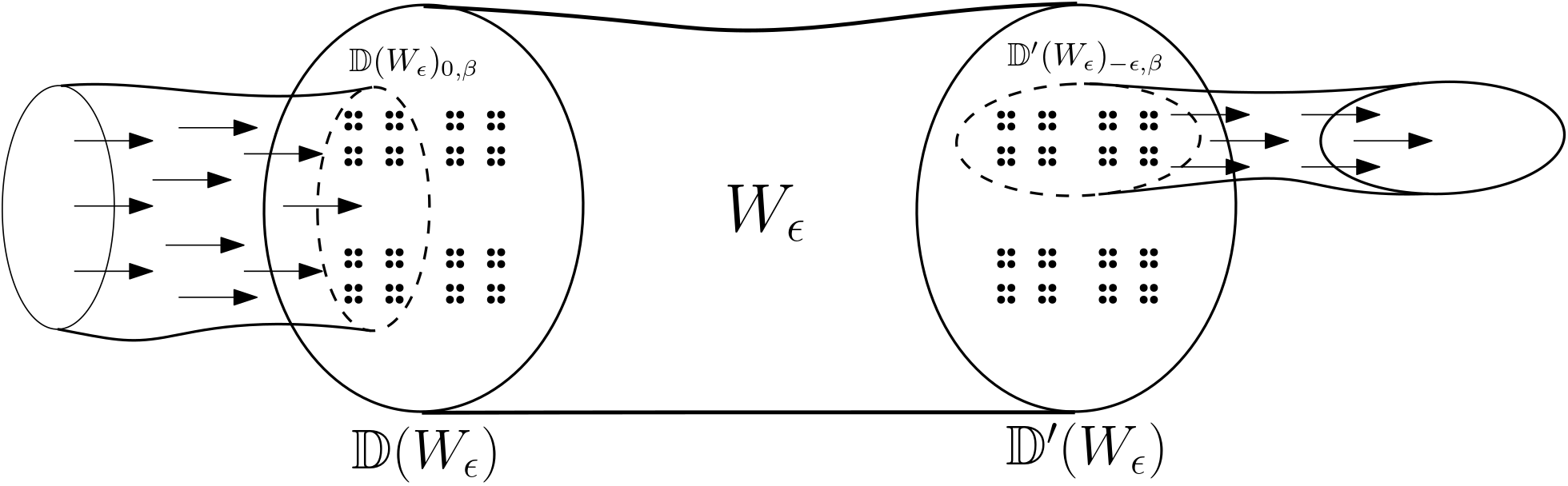}
    \caption{Gluing of the bordism $W_\epsilon$ with two small cylinders.}
    \label{fig:gluing-short}
\end{figure}
    
Furthermore, since the vector field $X_\epsilon$ on $W_\epsilon$ is trivial in a neighbourhood of the boundary, it can be extended to $\widehat{W}_\epsilon$ by taking the vector field $\partial_z$ on the glued cylinders.

Now, the resulting space $\widehat{W}_\epsilon$ can be glued to the incoming disk $\mathbb{D}$ and to an outgoing disk $\mathbb{D}'$ to represent the read-write-shift operation $(\alpha, \beta, \epsilon)$. This is done by gluing the open set $\mathbb{D}_{0, \alpha}$ of $\mathbb{D}$ to the incoming boundary of $\widehat{W}_\epsilon$, which is $\mathbb{D}(W_\epsilon)_{0, \beta}$ and coincides with $\mathbb{D}_{0, \alpha}$ up to traslation. Analogously, we glue the open set $\mathbb{D}'(W_\epsilon)_{-\epsilon, \beta}$ of the outgoing boundary of $\widetilde{W}_\epsilon$ to the open set $\mathbb{D}_{-\epsilon, \beta}'$. Since the vector field is trivial in a boundary of the gluing loci, it naturally extends to a smooth vector field on the resulting space. The resulting bordism is shown in Figure \ref{fig:gluing}.

\begin{figure}[h]
    \centering
    \includegraphics[width=0.7\linewidth]{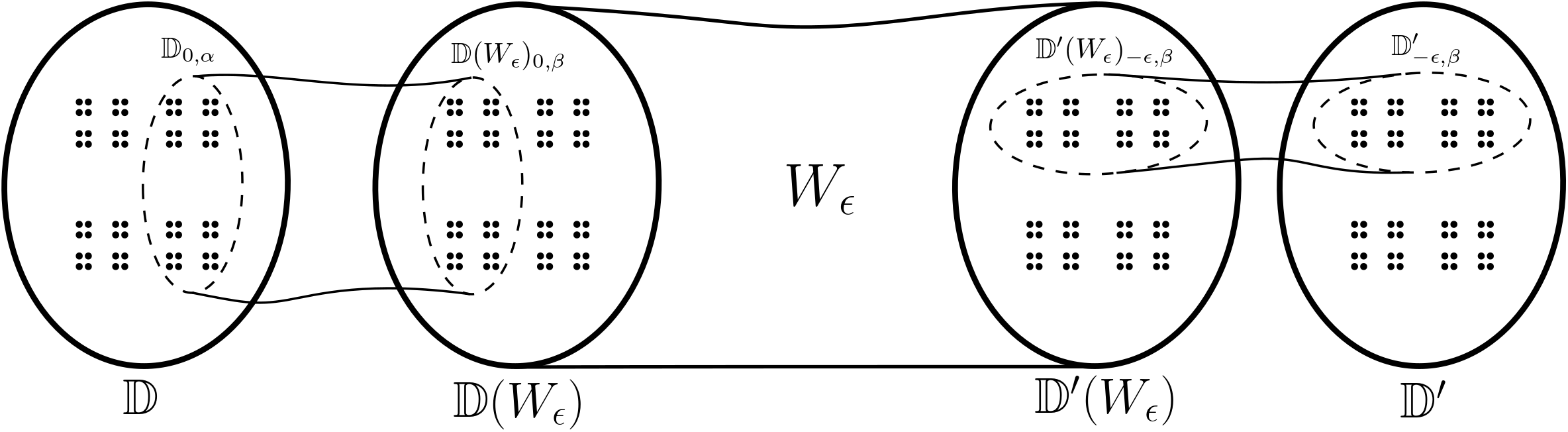}
    \caption{Gluing of the bordism $W_\epsilon$ with the disks $\mathbb{D}$ and $\mathbb{D}'$.}
    \label{fig:gluing}
\end{figure}
    
The flow of the resulting CW-complex performs exactly the read-write-shift operation $(\alpha, \beta, \epsilon)$. Indeed, since only the open set $\mathbb{D}_{0,\alpha} \subseteq \mathbb{D}$ is glued, the operation only affects to those tapes $t$ with $t_0 = \alpha$. This open set $\mathbb{D}_{0,\alpha}$ is glued to $\mathbb{D}(W_{\epsilon})_{0,\beta}$, having the effect of substituting the symbol $\alpha$ at position $0$ by the symbol $\beta$. In other words, under this gluing, a tape $t \in \mathbb{D}_{0,\alpha} \cap \Lambda$, i.e., with $t_0 = \alpha$, corresponds to another tape $t' \in \mathbb{D}(W_\epsilon)_{0,\beta} \cap \Lambda$ with $t'_n = t_n$ for $n \neq 0$ and $t'_0 = \beta$. Now, given such a tape $t' \in \mathbb{D}(W_\epsilon)_{0,\beta} \cap \Lambda$, the effect of the flow of $W_{\epsilon}$ is to send it to the tape $s_\epsilon(t') \in \mathbb{D}'(W_\epsilon) \cap \Lambda$. Concretely, since $t'_0 = \beta$, then $s_\epsilon(t')_{-\epsilon} = \beta$ and thus $s_\epsilon(t') \in \mathbb{D}'(W_\epsilon)_{-\epsilon, \beta}$. This is exactly the open set that is glued identically with the final disk $\mathbb{D}'$. Furthermore, the flow obtained in the resulting CW-complex is volume preserving, because the flows of each one of the basic bordisms that are glued are volume preserving as well.

\begin{proposition}\label{P:basicbord}
    For any $\epsilon \in \{\pm 1\}$, symbols $\alpha, \beta \in \mathcal{A}$, and disks $\mathbb{D}$ and $\mathbb{D}'$, there exists a dynamical bordism endowed with a volume-preserving vector field that is trivial in a neighbourhood of the boundary such that the flow performs the read-write-shift operation $(\alpha, \beta, \epsilon)$ from the disk $\mathbb{D}$ to the disk $\mathbb{D}'$. 
\end{proposition}

\subsection{Computational thickening}\label{sec:computational-thickening}

Consider a partial recursive function $f: \NN \dashrightarrow \NN$. By \cite{kleene1936general}, there exists a Turing machine $M = (Q, q_0, Q_{\textrm{halt}}, \delta)$ computing $f$. Furthermore, by Proposition \ref{prop:reversible-tm}, we can suppose that $M$ is reversible. In this section, we will prove that there exists a dynamical bordism $(W, X): \mathbb{D} \to \mathbb{D}$ such that $\psi_{(W,X)}: \Lambda \to \Lambda$ is the output function of $M$ (recall that the map $\psi_{(W,X)}$ was introduced in Section~\ref{sec:pf-bordisms}).

Now, consider the finite state machine $\mathcal{G}_M$ of $M$ as described in Section \ref{sec:graph-automata}. Without loss of generality, we can modify the state machine $\mathcal{G}_M$ to assume that it satisfies the following properties:
\begin{itemize}

    \item There is a single halting state node $q_f$. If there are more than one halting state nodes, then create a new halting state node and relabel all the former halting states as regular states. Also, create a new node with incoming edges of the form $(s, s, +)$ for all $s \in \mathcal{A}$, connecting all the formerly halting states with the new node, and outgoing vertices of the form $(s, s, -)$ connecting it with the new halting state, as in Figure \ref{fig:creating-halting-state}.

    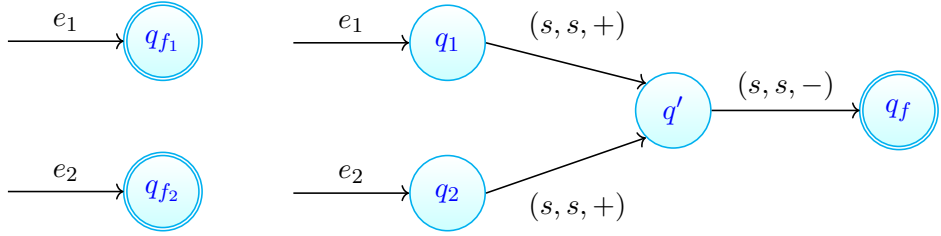
\begin{figure}[h!]
    \centering
    \begin {tikzpicture}[-latex ,auto ,node distance =2 cm and 3cm ,on grid , semithick , state/.style ={ circle ,top color =white , bottom color = processblue!20 ,
draw,processblue , text=blue , minimum width =1 cm}]
\node[state, accepting] (A){$q_{f_1}$};
\node[state, accepting] (B)[below =of A]{$q_{f_2}$};

\draw[<-] (A.west)  -- node[above]{$e_1$} ++(-4em,0em);
\draw[<-] (B.west)  -- node[above]{$e_2$} ++(-4em,0em);
\end{tikzpicture}
\hspace{1cm} 
    \begin {tikzpicture}[-latex ,auto ,node distance =2 cm and 3cm ,on grid , semithick , state/.style ={ circle ,top color =white , bottom color = processblue!20 ,
draw,processblue , text=blue , minimum width =1 cm}]
\node[state] (A){$q_{1}$};
\node[state] (B)[below =of A]{$q_{2}$};
\node[state] (C)[below right =of A, yshift = 1.1cm]{$q'$};
\node[state, accepting] (D)[right =of C]{$q_f$};

\draw[<-] (A.west)  -- node[above]{$e_1$} ++(-4em,0em);
\draw[<-] (B.west)  -- node[above]{$e_2$} ++(-4em,0em);
\path[->] (A.east) edge node[pos=0.2, above right] {$(s,s, +)$} (C.north west);
\path[->] (B.east) edge node[pos=0.2, below right] {$(s,s, +)$} (C.south west);
\path[->] (C.east) edge node[above] {$(s,s, -)$} (D.west);

\end{tikzpicture}
    \caption{Creating a single halting state node. On the left hand side we have a state machine with two halting states, and on the right hand side the result of fusing them into a single halting state.}
    \label{fig:creating-halting-state}
\end{figure}
    \item The starting state $q_0$ has no incoming edges. Otherwise, create a new state $q_0'$ receiving all the incoming edges of $q_0$, and connect them through an auxiliary vertex as in the previous case.
\end{itemize}

After these modifications, the proof will consist in `thickening' the graph $\mathcal{G}_M$ to turn it into a bordism, with disks representing the states and tubes capturing the dynamics of the arrows. In these disks, we will encode the current state of the tape $t$ of the Turing machine $M$ as a point of $\Lambda \subseteq \mathbb{D}$, as explained in Section \ref{sec:cantor-encoding}.

Now, given a reversible finite state machine $\mathcal{G}_M$ satisfying the properties above, we are going to construct a $3$-dimensional CW-complex $\widetilde{W}_{M}$ with a volume-preserving vector field $\widetilde{X}_M$ representing the dynamics of $\mathcal{G}_M$. This CW-complex is constructed as follows.
\begin{enumerate}
    \item We place a disk for each vertex of $\mathcal{G}_M$. We denote the disk associated to the state $q\in Q$ by $\mathbb{D}(q)$.
    \item Consider an edge in $\mathcal{G}_M$ between vertices $q$ and $q'$ and with label $(\alpha, \beta, \epsilon)$, where $\alpha$ is the read symbol, $\beta$ is the written symbol, and $\epsilon = \pm 1$ is the shift. We glue the bordism $W_\epsilon$ to $\mathbb{D}(q)$ and $\mathbb{D}(q')$ to represent the operation $(\alpha,\beta, \epsilon)$ as in Section \ref{sec:representation-basic}, cf. Proposition~\ref{P:basicbord}.
\end{enumerate}

Notice that the loops in $\mathcal{G}_M$ are also glued according to (2). Observe that the glued tubes are disjoint at their incoming boundary, since one transition occurs for every symbol $\alpha$ and thus exactly one tube is glued to each open set $\mathbb{D}_{0,\alpha}(q)$ for every state $q$. Analogously, the glued tubes are also disjoint at their outgoing boundary since $M$ is reversible, so every open set $\mathbb{D}_{\pm 1,\beta}(q')$ receives at most one bordism for every state $q'$.

This observation, together with the fact that the vector field is trivial around every boundary, implies that the vector fields on each of the bordisms glue together to define a smooth volume-preserving vector field $\widetilde{X}_M$ on $\widetilde{W}_M$. To construct the desired vector field on the homotopically equivalent bordism $W_M$, we use an auxiliary smooth manifold with boundary $\widehat W_M$, which is simply a small thickening of $\widetilde W_M$ where we round off the corners. It is enough to ``smooth out'' $\widetilde{X}_M$ only in a neighborhood of the disks, which allows us to define a volume preserving vector field $\widehat{X}_M$ on $\widehat W_M$ as the natural extension of $\widetilde{X}_M$ using that it is trivial around the boundary disks.

Next, we claim that the vector field $\widehat{X}_M$ can be extended to the whole bordism $W_M$ as a volume preserving vector field $X_M$. Indeed, since $\widehat X_M$ is volume preserving and it has zero flux across the boundary of the smooth manifold $\widehat W_M$ (by Gauss theorem), Poincar\'e's lemma~\cite[Corollary 3.2.4]{schwarz2006hodge} implies that there exists a smooth vector field $Y$ such that $\widehat X_M=\textup{curl}(Y)$.

Therefore, taking any smooth extension of the vector field $Y$, yields a volume preserving extension $X_M$ of $\widehat{X}_M$ to $W_M$. It is clear (e.g., by gluing with a suitable volume preserving vector field defined in a neighborhood of $\partial W_M$) that $X_M$  can be assumed to be tangent to the internal boundary $\partial'W_M \subseteq W_M$ and transverse to the other disk boundaries. Moreover, by genericity, we can slightly perturb $X_M$ on $W_M\backslash \widetilde W_M$ so that its zero set consists of finitely many points. The resulting dynamical bordism $(W_M, X_M)$ is obviously clean in the sense of Definition \ref{defn:clean}, since the retract $(\widetilde{W}_M, \widetilde{X}_M)$ was precisely constructed by gluing finitely many basic bordisms (which are computable by construction).

{ Figure \ref{fig:thickeningCW}} depicts the smooth bordism $(W_M, X_M)$ obtained after the thickening of the original CW-complex $(\widetilde{W}_M, \widetilde{X}_M)$. In this figure, red disks represent outgoing boundaries of the basic bordisms, whereas blue disks are incoming boundaries. Observe that the red disks are disjoint, since only one transition occurs for every input, and similarly the blue disks are disjoint since the Turing machine is reversible. However, red disks are not disjoint from blue disks, but they are glued at different sides of the bigger black disks. Indeed, the computational dynamics arise precisely in these intersections between incoming and outgoing boundaries.

\begin{figure}[h]
    \centering
    \includegraphics[width=0.23\linewidth]{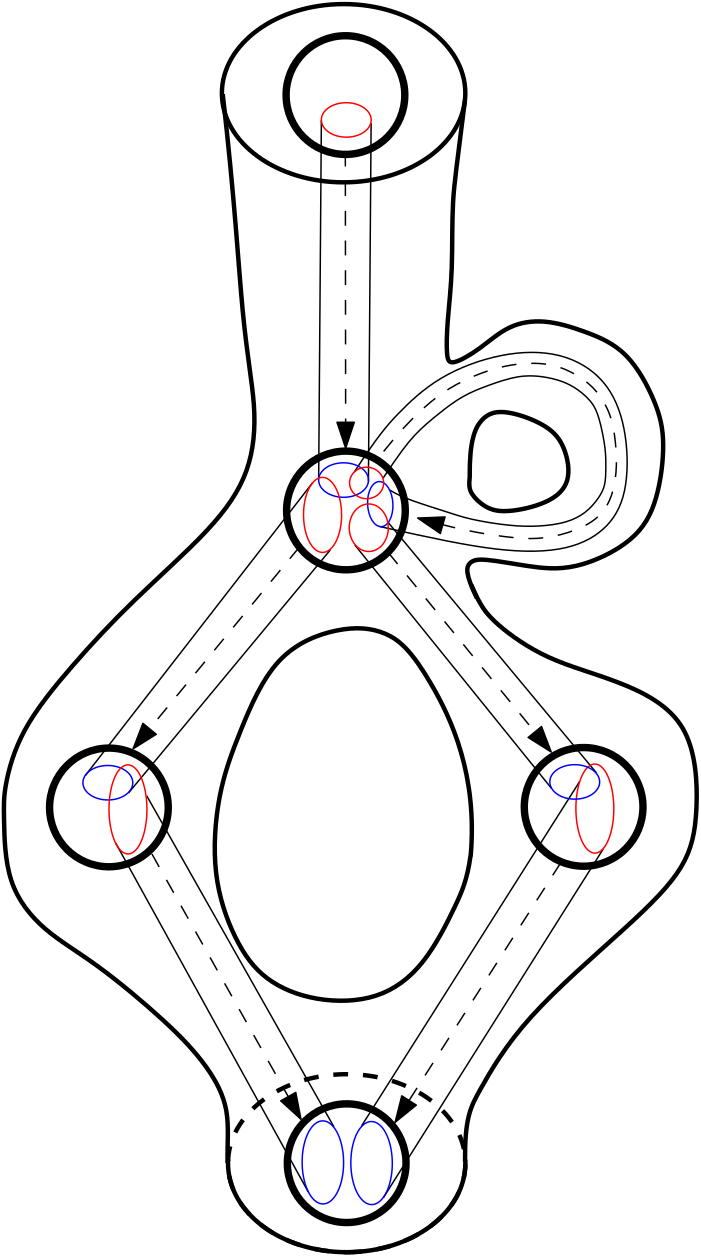}
    \caption{Resulting smooth bordism $(W_M, X_M)$ after the thickening process with the core CW-complex $(\widetilde{W}_M, \widetilde{X}_M)$ inside.}
    \label{fig:thickeningCW}
\end{figure}
    
By construction, $W_M$ is a bordism with boundaries $\mathbb{D}(q_0)$ and $\mathbb{D}(q_f)$ corresponding to the disks attached to the (unique) starting and halting states of $M$. The flow of $X_M$ in $W_M$ satisfies the following properties.

First, given a state $q$, the points belonging to the set $\Lambda \cap \mathbb{D}(q)$ encode the tape state of the Turing machine $M$ when placed at state $q$ as a two-sided binary sequence. Furthermore, suppose that $x \in \Lambda \cap \mathbb{D}(q)$ represents a tape $t$. Let $\alpha = t_0$ be the initial symbol of the tape and $(q', \beta, \epsilon) = \delta(q, \alpha)$ the resulting state $q'$, symbol $\beta$ to write and shifting $\epsilon$ of the transition function of $M$ at this state. This means that there is an edge of the form $(\alpha, \beta, \epsilon)$ that joins $q$ and $q'$ in the state machine of $M$. Then $x$ is glued to a bordism between $\mathbb{D}(q)$ and $\mathbb{D}(q')$ that represents this transition and, as shown in Section \ref{sec:representation-basic}, under this flow it reaches a unique point $x' \in \Lambda \cap \mathbb{D}(q')$ representing the resulting tape after executing the transition.

Now, let $x \in \Lambda \cap \mathbb{D}(q)$ be a point representing a tape $t$ that encodes a natural number $n \in \NN$ in binary. Let $t^0 = t, t^1, t^2, \ldots$ be the sequence of tapes obtained after successively executing $M$, which passes through states $q_0, q_1, q_2, \ldots$ subsequently. By the previous observation, this uniquely corresponds to a sequence of points encoding the tapes $x_0 = x, x_1, x_2, \ldots$, with $x_i \in \Lambda \cap \mathbb{D}(q_i)$ for all $i$, such that the flow of $X_M$ connects them exactly in that order. Then we have two options:
        \begin{itemize}
            \item If the Turing machine $M$ eventually stops, then $q_N = q_f$ for some $N$ large enough, and thus $x_N \in \Lambda \cap \mathbb{D}(q_f)$ is in the outgoing boundary of $W_M$ and encodes the resulting tape $t^N$. By the construction of the Turing machine described in Section \ref{sec:reversible-Turing}, the non-negative part of $t^N$ encodes the output $f(n)$ of the partial recursive function $f$ on $n$, whereas the negative part encodes the input $n$ itself.
            \item If $M$ does not halt, then the sequence $x_i$ is infinite with $x_i \not\in\mathbb{D}(q_f)$ for all $i$. Hence, the flow of $X_M$ traps the orbit inside the bordism.
        \end{itemize}

According to the previous observations, if we consider the partial function
$$
    \psi_{(W_M, X_M)}: \NN^2 \dashrightarrow \NN^2
$$
induced by the vector field $X_M$, we have that $\psi_{(W_M, X_M)}$ encodes the output function of the Turing machine associated to the initial partial recursive function $f: \NN \dashrightarrow \NN$. Hence, $Z_0(W_M, X_M) = f$, finishing the proof of Theorem \ref{thm:main}.

\section{Categorical interpretation: Topological Kleene Field Theories}\label{sec:categorical-interpretation}

The results of this work can be equivalently reformulated in the framework of category theory, where new and interesting perspectives arise. In particular, we will use the language of $2$-categories, also known as bicategories. Intuitively, a $2$-category is a collection of objects $x$, morphisms $f: x \to x'$ between objects (called $1$-morphisms) and morphisms $\alpha: f \Rightarrow f'$ between $1$-morphisms (called $2$-morphisms). For a complete introduction to $2$-categories, see \cite{benabou1967introduction}.

Notice that given $m \geq 1$, we can consider the disjoint union $\NN^{\sqcup m} = \NN \times \{1, \ldots, m\}$ of $m$ copies of $\NN$. For convenience, we also set $\NN^{\sqcup 0} = \emptyset$. Using the natural inclusion, $\NN^{\sqcup m}$ can be seen as a subset of $\NN^m$. With this notion at hand, first we can consider the \emph{$2$-category $\PF$ of partial functions} given by:
\begin{itemize}
    \item Objects: The objects of $\PF$ are the non-negative integers $\ZZ_{\geq 0}$. 
    \item Morphisms: A morphism between objects $m$ and $m'$ is a partial function $f: \NN^{\sqcup m} \dashrightarrow \NN^{\sqcup m'}$. Composition is given by composition of partial functions, with the identity map as the identity morphism.
    \item $2$-Morphisms: Given two morphisms $f: \NN^{\sqcup m} \dashrightarrow \NN^{\sqcup m'}$ and $g: \NN^{\sqcup m} \dashrightarrow \NN^{\sqcup m'}$, a unique $2$-morphism $f \Rightarrow g$ between them exists if $f$ extends $g$, i.e., if $D_{g} \subseteq D_{f}$ and $f|_{D_g} = g$.
\end{itemize}
Notice that $\PF$ is naturally equipped with a monoidal symmetric structure, where $m_1 \otimes m_2 = m_1 + m_2$ on objects. Additionally, given morphisms $f_1: \NN^{\sqcup m_1} \dashrightarrow \NN^{\sqcup m_1'}$ and $f_2: \NN^{\sqcup m_2} \dashrightarrow \NN^{\sqcup m_2'}$ then $f_1 \otimes f_2: \NN^{\sqcup m_1 + m_2} \dashrightarrow \NN^{\sqcup m_1' + m_2'}$ is the partial function given by $f_1$ on the first $m_1$ copies (with image in the first $m_1'$ copies) and $f_2$ on the last $m_2$ copies (and taking values in the last $m_2'$ copies).

Furthermore, inside $\PF$ we find an interesting wide $2$-subcategory, namely, the \emph{category $\PRF$ of partial recursive functions}. Again, the objects of $\PRF$ are the non-negative integers $\ZZ_{\geq 0}$, a morphism between them is a partial recursive function $f: \NN^{\sqcup m} \dashrightarrow \NN^{\sqcup m'}$, and $2$-morphisms are again restrictions of the domain. Notice that compositions of partial recursive functions are again partial recursive functions, showing that $\PRF$ is indeed a subcategory of $\PF$, and it is straightforward to check that the monoidal symmetric structure of $\PF$ also restricts to $\PRF$.

In the same spirit, fixed $n \geq 2$, we can consider the \emph{category $\Bord_n^{\textup{dy}}$ of clean dynamical $n$-bordisms} as the category given as follows:
\begin{itemize}
    \item Objects: An object is a triple $(M, X, \mathbf{\iota})$ of an $n$-dimensional marked dynamical germ. 
    \item Morphism: A morphism $(M_1, X_1, \bm{\iota}_1) \to (M_2, \bm{\iota}_2)$ between marked dynamical germs is an equivalence class of clean dynamical bordisms $(W, X)$ between them. Two such bordisms $(W, X)$ and $(W',X')$ are declared to be equivalent if there exists a boundary preserving diffeomorphism $F: W \to W'$ such that $F_*(X) = X'$.
    \item Composition is given by gluing of bordisms and vector fields along the common boundary.
\end{itemize}
Additionally, by considering only identity $2$-morphisms between bordisms, $\Bord_n^{\textup{dy}}$ can be seen as a $2$-category.

\begin{remark}
    The gluing $W \cup_{X_2} W'$ of two bordisms $(W, X): (M_1, X_1) \to (M_2, X_2)$ and $(W',X'): (M_2, X_2) \to (M_3, X_3)$ does not have a unique smooth structure. In the same vein, even though $X|_{X_2} = X'_{X_2}$, this does not mean that the vector fields $X$ and $X'$ glue to give a smooth vector field. However, by choosing a collaring around $M_2$ and smoothing-out $X$ and $X'$ along this collaring, there is a way of defining a smooth gluing with a smooth vector field. This smoothing-out is not unique in general, but it is unique up to boundary-preserving diffeomorphism, so composition is well-defined in $\Bord_n^{\textup{dy}}$.
\end{remark}

Recall that a functor $F: \mathbf{C} \to \mathbf{D}$ between $2$-categories is an assignment satisfying the same properties as a regular functor between usual categories except that, instead of preserving composition on the nose, there may exist a natural family of $2$-morphisms in $\mathbf{D}$ such that $F(f \circ g) \Rightarrow F(f) \circ F(g)$ for all $1$-morphisms $f$ and $g$ of $\mathbf{C}$.

As constructed in Section \ref{sec:pf-bordisms}, given a dynamical bordism $(W, X): (M, X, \bm{\iota}) \to (M', X', \bm{\iota}')$, we can associate to it a partial function
$$
    Z_0(W, X): \NN^{\sqcup m} \dashrightarrow \NN^{\sqcup m'},
$$
where $m$ and $m'$ are the number of connected components of $M$ and $M'$, respectively. We can promote this assignment to objects by setting $Z_0(M, X, \bm{\iota}) = m$, if $M$ has $m$ connected components. A straightforward check using the description of Section \ref{sec:pf-bordisms} shows the following result.

\begin{proposition}
    The assignment
    $$
        Z_0: \Bord_n^{\textup{dy}} \to \PF
    $$
    is a well-defined monoidal symmetric functor of $2$-categories.
\end{proposition}

\begin{remark}
Notice that the $2$-category structure plays a crucial role here. Given bordisms $(W, X): (M, X, \bm{\iota}) \to (M', X', \bm{\iota}')$ and $(W', X'): (M', X', \bm{\iota}') \to (M'', X'', \bm{\iota}'')$, the uniqueness of the flow of a smooth vector field shows that $Z_0\left((W', X')\circ (W, X)\right)$ is a partial function extending the composition $Z_0(W', X')\circ Z_0(W, X)$. By definition, this exactly means that there exists a $2$-morphism $Z_0(W', X')\circ Z_0(W, X) \Rightarrow Z_0\left((W', X')\circ (W, X)\right)$, but this is not invertible in principle.
\end{remark}

\begin{definition}
    The functor 
$$
Z_0: \Bord_n^{\textup{dy}} \longrightarrow \PF
$$
    is called the \emph{pseudo-Topological Kleene Field Theory}, pseudo-TKFT for short.
\end{definition}

\begin{definition}\label{defn:TKFT}
    A \emph{Topological Kleene Field Theory}, TKFT for short, is a monoidal subcategory $\mathcal{B} \subseteq \Bord_n^{\textup{dy}}$ such that $Z_0$ restricted to $\mathcal{B}$ factorizes as a full and surjective on objects monoidal functor
    $$
        Z: \mathcal{B} \longrightarrow \PRF.
    $$
    If such a functor exists, we shall say that partial recursive functions are quantizable by means of the category $\mathcal{B}$.
\end{definition}

\begin{remark}
    Definition \ref{defn:TKFT} requires three conditions to be fulfilled. The first one, that $Z_0$ restricts to a functor to $\PRF$, can be spelled out as that we have a commutative diagram
    $$
        \xymatrix{ \Bord_n^{\textup{dy}} \ar[r]^{Z_0} & \PF \\ \mathcal{B} \ar[r]_{Z} \ar@{^{(}->}[u] & \PRF \ar@{^{(}->}[u] }
    $$
    The functor $Z$ must be full. This means that for any object $(M, X, \bm{\iota})$ and $(M', X', \bm{\iota}')$, the assignment
    $$
        Z: \Hom_{\mathcal{B}}((M, X, \bm{\iota}), (M', X', \bm{\iota}')) \to \Hom_{\PRF}(Z(M, X, \bm{\iota}), Z(M', X', \bm{\iota}'))
    $$
    is subjective or, in other words, any partial recursive function is representable through a dynamical bordism. Finally, the condition of being subjective on objects reduces to the fact that $\mathcal{B}$ contains at least one connected object.
\end{remark}

\begin{remark}\label{rmk:category-b0}
    A way of constructing the category $\mathcal{B}$ of Definition \ref{defn:TKFT} is by using a certain `core' subcategory $\mathcal{B}_0$ of $\mathcal{B}$. This category $\mathcal{B}_0$ contains a single object $M = (M, X, \bm{\iota})$, with $M$ connected. It must also have the following properties: \begin{itemize}
        \item For any partial recursive function $f: \NN \dashrightarrow \NN$, there exists an endobordism $(W, X): M \to M$ such that $Z_0(W, X) = f$.
        \item Any endobordism $(W, X): M \to M$ of $\mathcal{B}_0$ satisfies that $Z_0(W, X)$ is a partial recursive function.
    \end{itemize}
    With this information at hand, $\mathcal{B}$ can be constructed as the monoidal closure of $\mathcal{B}_0$, i.e.\ the category made of disjoint unions of objects and bordisms of $\mathcal{B}_0$. Since $Z_0(M^{\sqcup n}) = 1^{\otimes n} = n$, we have that $Z = Z_0|_{\mathcal{B}}$ is subjective on objects and, since $Z_0|_{\mathcal{B}_0}$ takes values in $\PRF$ and $\PRF$ is monoidal, the functor $Z$ also takes values in $\PRF$. Finally, by the very construction of $\mathcal{B}_0$, the functor $Z$ is full.
\end{remark}

In this categorical language, Theorem \ref{thm:main} can be rephrased as follows.

\begin{theorem}\label{thm:fundamental-lemma}
    Topological Kleene Field Theories exist.
\end{theorem}

\begin{proof}
By Theorem \ref{thm:main}, given a partial recursive function, there exists a dynamical endobordism of the standard disk representing it. Now, consider the category $\mathcal{B}_0$ with $\mathbb{D}$ as the only object. For the morphisms, pick a Turing machine $M_f$ computing a partial recursive function $f: \NN \dashrightarrow \NN$ for any $f$, and take as morphisms the compositional closure collection of the bordisms $(W_{M_f}, X_{M_f})$ satisfying $Z_0(W_{M_f}, X_{M_f}) = f$. By construction, $\mathcal{B}_0$ satisfies the conditions of Remark \ref{rmk:category-b0} and thus defines a TKFT, proving the result.
\end{proof}

A TKFT provides an effective method of computing through a dynamical system. Indeed, if $Z: \mathcal{B} \to \PRF$ is a TKFT, then given a partial recursive function $f: \NN \dashrightarrow \NN$, there exists a bordism $(W, X)$ of $\mathcal{B}$ such that $Z(W, X) = f$. This has many implications:
\begin{itemize}
    \item The partial recursive function $f$ can be effectively computed by simulating the flow of $X$ in $W$.
    \item The dynamics of the vector fields in $\mathcal{B}$ manifest a special kind of complexity, in the sense that they are Turing-complete.
    \item Since there exist universal partial recursive functions, then $\mathcal{B}$ also contains a universal bordism $(W_u, X_u)$. Somehow, this means that the dynamics of any other bordism in $\mathcal{B}$ is governed by those of $(W_u, X_u)$.
    \item The dynamics in $\mathcal{B}$ are undecidable. That is, given $(W, X): (M_1, X_1, \bm{\iota}_1) \to (M_2, X_2, \bm{\iota}_2)$ a morphism of $\mathcal{B}$, there exists no algorithm to determine whether the flow of a point starting at $M_1$ will hit $M_2$ or not.
\end{itemize}

{
We finish this section by pointing out that, to the best of our knowledge, it is not known whether the image of the functor $Z_0: \Bord_n^{\textup{dy}} \to \PF$ is always a partial recursive function. In other words, the clean condition on bordisms might force the function represented by a clean dynamical bordism to be computable. This would imply that the subcategory $\mathcal{B}$ of Definition \ref{defn:TKFT} can be taken to be the full $\Bord_n^{\textup{dy}}$.

\begin{theorem}\label{thm:clean-kleen}
    The reaching function of any clean dynamical bordism is a partial recursive functor. In other words, there exists a monoidal symmetric functor of $2$-categories
    $$
        Z: \Bord_n^{\textup{dy}} \to \PRF,
    $$
    with values in the category of partial recursive functions.
\end{theorem}
}

\section{Conclusions}\label{sec:conclusions}

{The work presented in this paper establishes a new connection between dynamical systems and computability by representing computable functions through the flows of vector fields on bordisms. These ideas raise interesting questions to be explored in future research. Indeed, as a result of this work, we have proven that TKFTs represent a fully functional model of computation, on the same footing as Turing machines and partial recursive functions. However, the dynamical nature of the construction enables new perspective to address classical problems of computation through a novel glass.

Firstly, TKFTs also preclude that some functions might be speeded up using dynamical bordisms. It is well known that some beyond-Turing methods, such as Turing machines with advice, can be represented as area-preserving diffeomorphisms of the disk (cf.\ \cite[Corollary 3.3]{cardona2024hydrodynamic}). In this sense, we expect that some ``unconventional'' diffeomorphisms of the disk may be used to speed-up the computation, the same way the advice accelerates the computation in super-Turing models. This emerging feature of clean dynamical bordisms would demonstrate that TKFT already exceeds the efficiency of Turing machines.

Directly related to this problem, the method developed in this paper links the complexity of a computable function with the topology of the bordism representing it. As is evident from the proof, a pair-of-pants, as a bordism between one disk and two disks, should be understood as an `if' instruction in the computation, whereas a nontrivial handle represents a `loop' construction.
In fact, the method of proof creates a bordism that has as strong deformation retract the graph underlying a Turing machine computing the function, so the homology invariants of the bordisms, and in particular its first Betti number, are forced to agree with the ones of the graph. In this direction, a question arises: is it possible to (algorithmically) construct a clean dynamical bordism computing a partial recursive function $f$ with strictly smaller first Betti number than that of any Turing machine computing $f$? {This is particularly interesting taking into account that TKFTs are equivalent to Turing machines, since using clean bordisms we will be able to compute exactly the same partial recursive functions, but potentially improving the complexity of the calculation.}

A positive answer to these questions would imply that dynamical bordisms offer a new framework for computability that is strictly more efficient than standard Turing machines, analogous to the way quantum computing provides a more efficient approach to computation. In this context, an alternative approach would be to employ stochastic vector fields on the bordism. In this framework, the flow would be probabilistic rather than deterministic, mimicking the stochastic nature of quantum computing. Indeed, in a certain sense, the topological nature of bordisms resembles the bifurcations observed in Feynman diagrams. Consequently, it could be anticipated that the capabilities of dynamical bordisms might match, if not surpass, those of quantum computing.

}

\bibliographystyle{abbrv} %
\bibliography{bibliography}

\printauthorinfo

\end{document}